\documentclass[11pt]{amsart}
\usepackage{amsthm,amsmath,amssymb}
\usepackage{stmaryrd,mathrsfs}
\usepackage{enumerate}
\usepackage[T1]{fontenc}
\usepackage{esint}
\usepackage{tikz}
\usepackage{hyperref}
\allowdisplaybreaks

\newcommand{\abs}[1]{|#1|}

\newcommand{\Norm}[2]{\|#1\|_{#2}}

\newcommand{\ave}[1]{\langle #1\rangle}

\newcommand{\BMO}[0]{\operatorname{BMO}}
\newcommand{\CMO}[0]{\operatorname{CMO}}

\newcommand{\supp}[0]{\operatorname{supp}}
\newcommand{\loc}[0]{\operatorname{loc}}

\newcommand{\testi}{{\mathcal S}}

\newcommand{\R}{\mathbb{R}}
\newcommand{\C}{\mathbb{C}}
\newcommand{\N}{\mathbb{N}}
\newcommand{\Z}{\mathbb{Z}}

\newcommand{\eps}[0]{\varepsilon}

\swapnumbers \numberwithin{equation}{section}

\theoremstyle{plain}
\newtheorem{theorem}[equation]{Theorem}
\newtheorem{proposition}[equation]{Proposition}

\newtheorem{lemma}[equation]{Lemma}

\theoremstyle{definition}
\newtheorem{definition}[equation]{Definition}

\theoremstyle{remark}
\newtheorem{remark}[equation]{Remark}

\makeatletter
\@namedef{subjclassname@2020}{%
  \textup{2020} Mathematics Subject Classification}
 \def\@textbottom{\vskip \z@ \@plus 1pt}
 \let\@texttop\relax
\makeatother

\makeatletter
\def\namedlabel#1#2{\begingroup
   \def\@currentlabel{#2}%
   \label{#1}\endgroup
}
\makeatother

\begin{document}

\title[Extrapolation of compactness]{Extrapolation of compactness on\\ weighted spaces: bilinear operators}

\author[T. \ Hyt\"onen and S. Lappas]{Tuomas Hyt\"onen and Stefanos Lappas}
\address{Department of Mathematics and Statistics, P.O.Box~68 (Pietari Kalmin katu~5), FI-00014 University of Helsinki, Finland}
\email{tuomas.hytonen@helsinki.fi}
\email{stefanos.lappas@helsinki.fi}





\maketitle


\begin{abstract}
In a previous paper, we obtained several ``compact versions'' of Rubio de Francia's weighted extrapolation theorem, which allowed us to extrapolate the compactness of linear operators from just one space to the full range of weighted Lebesgue spaces, where these operators are bounded. In this paper, we study the extrapolation of compactness for bilinear operators in terms of bilinear Muckenhoupt weights. As applications, we easily recover and improve earlier results on the weighted compactness of commutators of bilinear Calder\'{o}n-Zygmund operators, bilinear fractional integrals and bilinear Fourier multipliers. More general versions of these results are recently due to Cao, Olivo and Yabuta (arXiv:2011.13191), whose approach depends on developing weighted versions of the Fr\'echet--Kolmogorov criterion of compactness, whereas we avoid this by relying on ``softer'' tools, which might have an independent interest in view of further extensions of the method.
\end{abstract}

\section{Introduction}

Rubio de Francia's weighted extrapolation theorem \cite{Rubio:factorAp} is one of the cornerstones of the modern theory of weighted norm inequalities. It enables one to deduce the {\em boundedness} of a given operator on $L^p(w)$ for all $1<p<\infty$ and all weights $w\in A_{p}(\R^d)$, provided this operator is bounded on $L^{p_0}(w_0)$ for some $1<p_0<\infty$ and all weights $w_0\in A_{p_0}(\R^d)$. Different versions of this extrapolation theorem are studied in \cite{CUMP:book}.

A multilinear Rubio de Francia extrapolation theorem of boundedness on weighted spaces was first established by Grafakos and Martell in \cite{GM} (see also the extension of this result in \cite{CM}). The main disadvantage of these results is that they treat each variable separately with its own Muckenhoupt class of weights and do not fully use the multilinear nature of the problem. In this direction, Li--Martell--Ombrosi \cite{LMO} (see also \cite{LMMOV,N} for further extensions to end-point cases) obtained a more satisfactory multilinear analogue of the Rubio de Francia's extrapolation theorem dealing with the multilinear $A_{\vec p}(\R^{md})$ classes introduced in \cite{KLOPTG}. We state here the bilinear version of their extrapolation result as follows (we will provide detailed definitions in the next section):

\begin{theorem}[\cite{LMO}, Corollary 1.5]\label{thm:extrapol. bdd.}
Let $\mathcal{F}$ be a collection of triplets $(f,f_1,f_2)$ of non-negative functions. Let $\vec{p}=(p_1,p_2)$ be exponents with $1\leq p_1,p_2<\infty$, such that given any $\vec{w}=(w_1,w_2)\in A_{\vec{p}}(\R^{2d})$, the inequality
\begin{equation*}
  \|f\|_{L^p(\nu_{\vec{w},\vec{p}})} \lesssim \prod_{i=1}^2 \|f_i\|_{L^{p_i}(w_i)}, 
\end{equation*}
holds for all $(f,f_1,f_2)\in\mathcal{F}$, where $\frac{1}{p}=\frac{1}{p_1}+\frac{1}{p_2}$ and $\nu_{\vec{w},\vec{p}}=\prod_{i=1}^2 w_i^{p/p_i}$. Then for all exponents $\vec{q}=(q_1,q_2)$ with $1<q_1,q_2<\infty$, and for all weights $\vec{v}=(v_1,v_2)\in A_{\vec{q}}(\R^{2d})$ the inequality
\begin{equation*}
  \|f\|_{L^q(\nu_{\vec{v},\vec{q}})} \lesssim \prod_{i=1}^2 \|f_i\|_{L^{q_i}(v_i)} 
\end{equation*}
holds for all $(f,f_1,f_2)\in\mathcal{F}$, where $\frac{1}{q}=\frac{1}{q_1}+\frac{1}{q_2}$ and $\nu_{\vec{v},\vec{q}}=\prod_{i=1}^2 v_i^{q/q_i}$.  
\end{theorem}

In a recent paper, we \cite{HL} first provided the extrapolation of {\em compactness} of a linear operator. Moreover, we obtained generalizations of the preceding compact extrapolation to the ``off-diagonal'' and limited range cases.

Inspired by the work above, we extend our results of \cite{HL} about the extrapolation of compactness to the following bilinear setting:

\begin{theorem}\label{thm:main 1}
Let $\Theta$ be a collection of ordered triples of Banach spaces $(Y_1,Y_2,Y)$, and let $T$ be a bilinear operator defined and 
\begin{equation*}
  \text{bounded}\quad T:Y_1\times Y_2\to Y\quad\text{for {\bf all}}\quad (Y_1,Y_2,Y)\in\Theta
\end{equation*}
and
\begin{equation*}
  \text{compact}\quad T:X_1\times X_2\to X\quad\text{for {\bf some}}\quad (X_1,X_2,X)\in\Theta.
\end{equation*}
Then $T$ is
\begin{equation*}
  \text{compact}\quad T:Z_1\times Z_2\to Z\quad\text{for {\bf all}}\quad (Z_1,Z_2,Z)\in\Theta
\end{equation*}
in each of the several cases of $\Theta$ involving weighted $L^p$ spaces as described in Theorem \ref{thm:main} of Section \ref{section 2}.
\end{theorem}

Shortly before our completion of this paper, this same result, even in a more general version covering higher order multilinearities and quasi-Banach spaces, was already announced by Cao, Olivo and Yabuta \cite{COY}, which gives these authors a priority to this result.
The overall relation of the present paper and \cite{COY} is a bit complicated, due to several subsequent versions of both works that were posted in the arXiv. As said, version 1 of \cite{COY} (Nov 2020) preceded ours, but did not provide a self-contained argument, since some results were quoted from a preprint of the same authors that was not publicly available. This was fixed in version 2 of \cite{COY} that was posted shortly after version 1 of the present work in Dec 2020. On the other hand, versions 1 and 2 of \cite{COY} did not treat the ``off-diagonal'' cases of extrapolation, which we covered since our version 1 (case \ref{(2)} of our Theorem \ref{thm:main}) but which was only added (in a more general form) in version 3 of \cite{COY} (Feb 2021). The newest version of \cite{COY} hence seems to supersede ours in all aspects, but there are a couple of features in our approach that still make it a worthwhile alternative:
\begin{itemize}
  \item As in the previous part \cite{HL} of this series, we have tried to make our approach as ``soft'' as possible, so that compactness is achieved by abstract means, without the need to describe concrete conditions for compactness in the weighted $L^p$ spaces. This is a main difference of our approach compared to all other works on compactness of operators on $L^p(w)$, including the recent \cite{COY}, where weighted versions of the {\em Fr\'echet--Kolmogorov compactness criterion} play a key role (see \cite[Lemma 2.9]{COY}, which extends \cite[Lemma 4.1]{XYY}). 
  \item Our result is still powerful enough to recover and improve several compactness results for bilinear commutators that were available before \cite{COY} (for applications see Sections \ref{section 6}--\ref{section 8}).
\end{itemize}

The paper is organized as follows: in Section \ref{section 2}, we recall some definitions about multilinear Muckenhoupt weights and we state in details our main result (see Theorem \ref{thm:main}). In Section \ref{section 3} we present the proof of Theorem \ref{thm:main} by collecting some previously known results and taking some auxiliary results for granted. Sections \ref{section 4} and \ref{section 5} are devoted to the proofs of these auxiliary results (see Proposition \ref{prop:main}). In Sections \ref{section 6}--\ref{section 8} we provide several applications of our main results. In particular, we obtain results for the commutators of
{\em bilinear Calder\'{o}n-Zygmund operators, bilinear fractional integral operators} and {\em bilinear Fourier multipliers}.

\subsection*{Notation} Throughout the paper, $C$ always denotes a positive constant that may vary from line to line but remains independent of the main parameters. We use the symbol $f\lesssim g$ to denote that there exists a positive constant $C$ such that $f\leq Cg$. The term cube always refers to a cube $Q\subset\R^d$ and $|Q|$ denotes its Lebesgue measure. We denote the average of $w$ over $Q$ as $\ave{w}_Q:=\abs{Q}^{-1}\int_Q w$ and $p'$ is the conjugate exponent to $p$, that is $p':=p/(p-1)$.

\subsection*{Acknowledgements} Both authors were supported by the Academy of Finland through the grant No. 314829. The second author wishes to thank his doctoral supervisor Prof. Tuomas Hyt\"onen for his suggestions and fruitful interaction. Also, the second author gratefully acknowledges financial support from the Foundation for Education and European Culture (Founders Nicos and Lydia Tricha).

\subsection*{Declarations of interest:} none.

\section{Preliminaries and the statement of the main result}\label{section 2}

We begin by recalling several definitions related to linear and multilinear Muckenhoupt weights. 

\begin{definition}[\cite{Muckenhoupt:Ap}]\label{def:Muchenhoupt weights}
A weight $w\in L_{\loc}^1(\R^d)$ is called a Muckenhoupt $A_p(\R^d)$ weight (or $w\in A_p(\R^d)$) if 
\begin{equation*}
\begin{split}
  &[w]_{A_p}:=\sup_Q\ave{w}_Q\ave{w^{-\frac{1}{p-1}}}_Q^{p-1}<\infty,\qquad 1<p<\infty, \\
  &[w]_{A_1}:=\sup_Q\ave{w}_Q\Norm{w^{-1}}{L^\infty(Q)}<\infty,\qquad p=1,
\end{split}
\end{equation*}
where the supremum is taken over all cubes $Q\subset\R^d$, and $\ave{w}_Q:=\abs{Q}^{-1}\int_Q w$.
A weight $w$ is called an $A_{p,q}(\R^d)$ weight (or $w\in A_{p,q}(\R^d)$) if
\begin{equation*}
  [w]_{A_{p,q}}:=\sup_Q\ave{w^q}_Q^{1/q}\ave{w^{-p'}}_Q^{1/p'}<\infty,\qquad 1<p\leq q<\infty,
\end{equation*}
where $p':=p/(p-1)$ denotes the conjugate exponent.
\end{definition}

\begin{definition}
Given a vector of weights $\vec w =(w_1,\ldots,w_m)$, and $\vec p =(p_1,\ldots,p_m)\in(0,\infty)^m$, we define
\begin{equation*}
  \nu_{\vec w,\vec p}:=\prod_{j=1}^m w_j^{p/p_j},\qquad
  \nu_{\vec w}:=\prod_{j=1}^m w_j.
\end{equation*}
\end{definition}

\begin{definition}[\cite{KLOPTG}]\label{def:multilinearAp}
Let $\vec{p}=(p_1,\dots,p_m)$ and $\frac{1}{p}=\sum_{j=1}^m\frac{1}{p_j}$ with $1\leq p_1,\dots,p_m<\infty$. We say that a vector of weights $\vec{w}=(w_1,\dots,w_m)$ satisfies the multilinear $A_{\vec{p}}(\R^{md})$ condition (or $\vec{w}\in A_{\vec{p}}(\R^{md})$) if 
\begin{equation*}
\begin{split}
  &[w]_{A_{\vec{p}}}:=\sup_Q\ave{\nu_{\vec{w},\vec{p}}}_Q^{\frac{1}{p}}\prod_{j=1}^m\ave{w_j^{1-p'_j}}_Q^{\frac{1}{p'_j}}<\infty,
\end{split}
\end{equation*}
where the supremum is taken over all cubes $Q\subset\R^d$, and $\ave{w_j}_Q:=\frac{1}{\abs{Q}}\int_Q w_j$.
\newline When $p_j=1$, $\ave{w_j^{1-p'_j}}_Q^{\frac{1}{p'_j}}$ is understood as $(\inf_Q w_j)^{-1}$.
\end{definition}

\begin{remark}
Note that if $m=1$, then $A_{\vec{p}}(\R^{md})$ is just the classical weight class $A_p(\R^d)$.
\end{remark}

\begin{definition}[\cite{Jiao}]
Let $m\ge1$ be an integer, $\vec{p}=(p_1,\dots,p_m)\in(0,\infty)^m$, $\frac{1}{p}=\sum_{j=1}^m\frac{1}{p_j}$, $s_j\in(0,p_j]$ ($1\leq j\leq m$) and $\frac{1}{s}=\sum_{j=1}^m\frac{1}{s_j}$. We say that a vector of weights $\vec{w}=(w_1,\dots,w_m)$ satisfies the multilinear $A_{\vec{p}/\vec{s}}(\R^{md})$ condition (or $\vec{w}\in A_{\vec{p}/\vec{s}}(\R^{md})$) if
\begin{equation*}
\begin{split}
  &[w]_{A_{\vec{p}/\vec{s}}}:=\sup_Q\ave{\nu_{\vec{w},\vec{p}}}_Q^{\frac{1}{p}}\prod_{j=1}^m\ave{w_j^{1-\big(\frac{p_j}{s_j}\big)^{'}}}_Q^{\frac{1}{s_j}-\frac{1}{p_j}}<\infty,
\end{split}
\end{equation*}
where the supremum is taken over all cubes $Q\subset\R^d$, and $\ave{w_j}_Q:=\frac{1}{\abs{Q}}\int_Q w_j$.
\newline When $p_j=s_j$, $\ave{w_j^{1-\big(\frac{p_j}{s_j}\big)^{'}}}_Q^{\frac{1}{s_j}-\frac{1}{p_j}}$ is understood as $(\inf_Q w_j)^{-\frac{1}{p_j}}$.
\end{definition}

\begin{remark}
When $s_1=\cdots=s_m=1$, $A_{\vec{p}/\vec{s}}(\R^{md})$ is just the weight class $A_{\vec{p}}(\R^{md})$ from Definition \ref{def:multilinearAp}. Note that we do not assign any independent meaning to the subscript ``$\vec p/\vec s$'' in $A_{\vec p/\vec s}$; the quotient line only serves a separator of the two vector indices $\vec p$ and $\vec s$.
\end{remark}

\begin{definition}[\cite{CX,Moen}]
Let $\vec{p}=(p_1,\dots,p_m)\in[1,\infty)^m$, $\frac{1}{p}=\sum_{j=1}^m\frac{1}{p_j}$ and $p^*$ be a number $1/m<p\leq p^*<\infty$. We say that a vector of weights $\vec{w}=(w_1,\dots,w_m)$ satisfies the multilinear $A_{\vec{p},p^*}(\R^{md})$ condition (or $\vec{w}\in A_{\vec{p},p^*}(\R^{md})$) if
\begin{equation*}
\begin{split}
  &[w]_{A_{\vec{p},p^*}}:=\sup_Q\ave{\nu_{\vec{w}}^{p^*}}_Q^{\frac{1}{p^*}}\prod_{j=1}^m\ave{w_j^{-p'_j}}_Q^{\frac{1}{p'_j}}<\infty,
\end{split}
\end{equation*}
where the supremum is taken over all cubes $Q\subset\R^d$, and $\ave{w_j}_Q:=\frac{1}{\abs{Q}}\int_Q w_j$.
\newline When $p_j=1$, $\ave{w_j^{-p'_j}}_Q^{\frac{1}{p'_j}}$ is understood as $(\inf_Q w_j)^{-1}$.
\end{definition}

\begin{remark}
When $m=1$, we note that $A_{\vec{p},p^*}(\R^{md})$ will degenerate into the classical weight class $A_{p,p^*}(\R^d)$.
\end{remark}

As we will work in the weighted setting, we consider weighted Lebesgue spaces
\begin{equation*}
  L^p(w):=\Big\{f:\R^d\to\C\text{ measurable }\Big|\ \Norm{f}{L^p(w)}:=\Big(\int_{\R^d}\abs{f}^p w\Big)^{1/p}<\infty\Big\}.
\end{equation*}

Our main result about the extrapolation of compactness for bilinear operators is as follows:

\begin{theorem}\label{thm:main}
Let $\Theta$ be a collection of ordered triples of Banach spaces $(Y_1,Y_2,Y)$, and let $T$ be a bilinear operator defined and 
\begin{equation}\label{condit. 1 of bdd.}
  \text{bounded}\quad T:Y_1\times Y_2\to Y\quad\text{for {\bf all}}\quad (Y_1,Y_2,Y)\in\Theta
\end{equation}
and
\begin{equation}\label{condit. 1 of cmp.}
  \text{compact}\quad T:X_1\times X_2\to X\quad\text{for {\bf some}}\quad (X_1,X_2,X)\in\Theta.
\end{equation}
Then $T$ is
\begin{equation*}
  \text{compact}\quad T:Z_1\times Z_2\to Z\quad\text{for {\bf all}}\quad (Z_1,Z_2,Z)\in\Theta
\end{equation*}
in each of the following cases, where $\alpha\geq 0$, $\vec s=(s_1,s_2)\in[1,\infty)^2$ and $\frac1s=\frac{1}{s_1}+\frac{1}{s_2}$:
\begin{enumerate}
  \item $\Theta$ consists of all triples $\Big(L^{q_1}(v_1), L^{q_2}(v_2), L^q(\nu_{\vec v,\vec q})\Big)$, where
\begin{equation*}
  \vec q=(q_1,q_2)\in(s_1,\infty)\times(s_2,\infty),\quad \frac1q=\frac1q_1+\frac1q_2<1,\quad \nu_{\vec{v},\vec{q}}=\prod_{j=1}^{2}v_j^\frac{q}{q_j}
\end{equation*}
and		
\begin{enumerate}
  \item\namedlabel{(1a)}{(1a)} $\vec v=(v_1,v_2)\in A_{\vec q/\vec s}(\R^{2d})$, or
  \item\namedlabel{(1b)}{(1b)} $\vec v=(v_1,v_2)\in A_{q_1/s_1}(\R^{d})\times A_{q_2/s_2}(\R^d)$.
\end{enumerate}
  \item\namedlabel{(2)}{(2)} $\Theta$ consists of all triples $\Big(L^{q_1}(v_1^{q_1}), L^{q_2}(v_2^{q_2}), L^{q^*}(\nu_{\vec v}^{q^*})\Big)$, where
\begin{equation*}
  \vec q=(q_1,q_2)\in(1,\infty)^2,\quad \frac1q=\frac1q_1+\frac1q_2\in(\alpha,\alpha+1),\quad
  \frac{1}{q^*}=\frac1q-\alpha,\;\;\;
  \nu_{\vec v}=\prod_{j=1}^{2}v_j 
\end{equation*}
and		
\begin{enumerate}
  \item[(c)]\namedlabel{(2c)}{(2c)} $\vec v=(v_1,v_2)\in A_{\vec q,q^*}(\R^{2d})$, or
  \item[(d)]\namedlabel{(2d)}{(2d)} $\vec v=(v_1,v_2)\in A_{q_1,\tilde q_1}(\R^{d})\times A_{q_2,\tilde q_2}(\R^d)$, where
  $\displaystyle \frac{1}{\tilde q_j}=\frac{1}{q_j}-\frac{\alpha}{2}$.
\end{enumerate}
\end{enumerate}
\end{theorem}

\begin{remark}
\begin{enumerate}
\item[(1)] Because of the extrapolation Theorem \ref{thm:extrapol. bdd.}, in the case \ref{(1a)} of Theorem \ref{thm:main}
it is enough to assume the boundedness \eqref{condit. 1 of bdd.} of a bilinear operator $T$ from $L^{q_1}(v_1)\times L^{q_2}(v_2)$ to $ L^q(\nu_{\vec v,\vec q})$ for some exponents $\vec q=(q_1,q_2)\in(s_1,\infty)\times(s_2,\infty)$ such that  $\frac1q=\frac1q_1+\frac1q_2<1$ and all weights $\vec v=(v_1,v_2)\in A_{\vec q/\vec s}(\R^{2d})$. The same observation applies to all the rest cases of Theorem \ref{thm:main}. Also, notice that in the case \ref{(2)} of Theorem \ref{thm:main} the point of the condition $\frac1q\in(\alpha,\alpha+1)$ is that we want that $q^*\in(1,\infty)$.
\item[(2)] The improvements \cite{LMMOV,N} of the bounded extrapolation Theorem \ref{thm:extrapol. bdd.} show that one can more generally allow for $1\leq p_1,p_2\leq\infty$ in the assumptions, and $1<q_1,q_2\leq\infty$ in the conclusions, as long as one of $q_i$ remains finite. In particular, under case \ref{(1a)} of Theorem \ref{thm:main}, the assumption \eqref{condit. 1 of bdd.} automatically bootstraps to a larger collection $\tilde\Theta\supsetneq\Theta$, which is defined like $\Theta$ in \ref{(1a)}, but with $\vec q\in(s_1,\infty]\times(s_2,\infty]\setminus\{(\infty,\infty)\}$. On the other hand, the compactness assumption \eqref{condit. 1 of cmp.}, which is made on {\em some} $(X_1,X_2,X)\in\Theta$, would obviously be weakened by allowing for $(X_1,X_2,X)\in\tilde\Theta$, and it is natural to ask whether this weakening of the assumptions (and hence strengthening of Theorem \ref{thm:main}) is still valid. We suspect ``yes'', but a justification of this would seem to require elaborating several parts of the argument, and hence we have decided to leave this extension outside the scope of the present work. We would like to thank an anonymous referee for raising this interesting question. (One might also ask whether one could achieve a more general conclusion allowing for {\em all} $(Z_1,Z_2,Z)\in\tilde\Theta$, but here we suspect that the answer is ``no'', or at least beyond any natural extension of the present approach. The reason is that our key Proposition \ref{prop:main} below is about realising the $Z_i$ spaces as interpolation spaces between some $X_i$ and $Y_i$ spaces, and this would not be possible if $Z_i$ was allowed to be an end-point $L^\infty$ space of the scale of $L^p$ spaces.)
\end{enumerate}
\end{remark}

\section{Proof of the main result via abstract interpolation}\label{section 3}

We collect the results from which the proof of Theorem \ref{thm:main} follows. 

Following \cite{CCM}, we say that $\bar{A}=(A_1,A_2)$ is a {\em Banach couple} if the two Banach spaces $A_j$ are continuously embedded in the same Hausdorff topological vector space. We write $A_j^{\circ}$ for the closure of $A_1\cap A_2$ in the norm of $A_j$. The Banach couple $A$ is said to be {\em regular} if $A_j^{\circ}=A_j$ for $j=1,2$.

We denote by ${\bf\mathcal{B}}(\bar{A}\times\bar{B},\bar{E})={\bf\mathcal{B}}(\bar{A}\times\bar{B},(E_1,E_2))$ the operators that satisfy the following:
\begin{equation*}
  \|T(a,b)\|_{E_j}\leq M_j\|a\|_{A_j}\|b\|_{B_j},\quad a\in A_1\cap A_2,\quad b\in B_1\cap B_2,\quad j=1,2,
\end{equation*}
where $T$ is a bilinear operator defined on $(A_1\cap A_2)\times(B_1\cap B_2)$ with values in $E_1\cap E_2$ and $M_j$ are positive constants.

Let $(\Omega,\mu)$ be a $\sigma$-finite measure space. We denote by $\mathcal{M}$ the collection of all (equivalence classes of) scalar-valued $\mu$-measurable functions on $\Omega$ that are finite $\mu$-almost everywhere. The space $\mathcal{M}$ becomes a complete metric space with the topology of convergence in measure on sets of finite measure.

We say that a Banach space $E$ of functions in $\mathcal{M}$ is a {\em Banach function space} if the following four properties hold:
\begin{enumerate}[(a)]
  \item Whenever $g\in\mathcal{M}$, $f\in E$ and $|g(x)|\leq|f(x)|$ $\mu$-a.e., then $g\in E$ and $\|g\|_E\leq\|f\|_E$.
  \item If $f_n\to f$  $\mu$-a.e., and if $\liminf_{n\to\infty}\|f_n\|_E<\infty$, then $f\in E$ and $\|f\|_E\leq\liminf_{n\to\infty}\|f_n\|_E$.
  \item For every $\Gamma\subseteq\Omega$ with $\mu(\Gamma)<\infty$, we have that $\chi_{\Gamma}\in E$.
  \item  For every $\Gamma\subseteq\Omega$ with $\mu(\Gamma)<\infty$ there is a constant $c_{\Gamma}>0$ such that $\int_{\Gamma}|f|d\mu\leq c_{\Gamma}\|f\|_E$ for every $f\in E$.
\end{enumerate}

Let $(\Gamma_n)$ be a sequence of $\mu$-measurable sets of $\Omega$. We put $\Gamma_n\to\emptyset \;\mu$-a.e. if the characteristic functions $\chi_{\Gamma_n}$ converges to $0$ pointwise $\mu$-a.e.

We say that a function $f\in E$ has {\em absolutely continuous norm} if $\|f\chi_{\Gamma_n}\|_E\\\to 0$ for every sequence $(\Gamma_n)$ satisfying that $\Gamma_n\to\emptyset \;\mu$-a.e. The space $E$ is said to have absolutely continuous norm if every function of $E$ has absolutely continuous norm.

If $E$ is a Banach function space then $E$ is continuously embedded in $\mathcal{M}$. Hence, if $E_1$ and $E_2$ are Banach function spaces on $\Omega$, we have that $(E_1,E_2)$ is a Banach couple.

Let $0<\theta<1$. If $E_1$ or $E_2$ has absolutely continuous norm, then
\begin{equation*}
  [E_1,E_2]_{\theta}=\{f\in\mathcal{M}:|f(x)|=|f_1(x)|^{1-\theta}|f_2(x)|^{\theta},f_j\in E_j,j=1,2\}
\end{equation*}
and
\begin{equation*}
  \|f\|_{[E_1,E_2]_{\theta}}=\inf\{\max(\|f_1\|_{E_1},\|f_2\|_{E_2}):|f|=|f_1|^{1-\theta}|f_2|^{\theta}\}.
\end{equation*}
In particular $[E_1,E_2]_{\theta}$ is a Banach function space.

Our main abstract tool is the following theorem of Cobos--Fern\' andez-Cabrera--Mart\' inez \cite{CCM}:

\begin{theorem}[\cite{CCM}, Theorem 3.2]\label{thm:CCM}
Let $\bar{A}=(A_1,A_2)$, $\bar{B}=(B_1,B_2)$ be Banach couples. Assume that $(\Omega,\mu)$ is a $\sigma$-finite measure space, let $\bar{E}=(E_1,E_2)$ be a couple of Banach function spaces on $\Omega$, let $0<\theta<1$ and $T\in{\bf\mathcal{B}}(\bar{A}\times\bar{B},\bar{E})$. If $T:A_1^{\circ}\times B_1^{\circ}\to E_1$ compactly and $E_1$ has absolutely continuous norm, then $T$ may be uniquely extended to a compact bilinear operator from $[A_1,A_2]_\theta\times[B_1,B_2]_\theta$ to $[E_1,E_2]_\theta$.
\end{theorem}

Examples of Banach function spaces that satisfy the assumptions of Theorem \ref{thm:CCM} are the unweighted Lebesgue spaces $L^p(\Omega)$ (see \cite[Corollary 3.3]{CCM}). For the present needs, we will only use Theorem \ref{thm:CCM} in the following special setting:

\begin{proposition}\label{prop:main}
Let $\Theta$ be a collection of ordered triples of Banach spaces, and let $(Y_1,Y_2,Y),(Z_1,Z_2,Z)\in\Theta$. Then there is another $(X_1,X_2,X)\in\Theta$ and $\gamma\in(0,1)$ such that
\begin{equation*}
  [X_j,Y_j]_\gamma=Z_j,\qquad[X,Y]_\gamma=Z
\end{equation*}
in each of the cases \ref{(1a)}, \ref{(1b)}, \ref{(2c)}, \ref{(2d)} of Theorem \ref{thm:main}.
\end{proposition}

We postpone the proof of Proposition \ref{prop:main} to Section \ref{section 5}. The verification of this proposition is the only component of the proof of Theorem \ref{thm:main} that requires actual computations, rather than just a soft application of known results. 

\begin{lemma}\label{lem:lemOk}
If $p_j\in[1,\infty)$ and $w_j$ are weights, then the spaces $A_j=B_j=E_j=L^{p_j}(w_j)$ $(j=1,2)$ satisfy all the assumptions of Theorem \ref{thm:CCM}.
\end{lemma}

\begin{proof}
By the dominated convergence theorem, it is easy to see that $A_j=B_j=E_j=L^{p_j}(w_j)$ have absolutely continous norm. The rest of the assumptions of Theorem \ref{thm:CCM} are satisfied by $A_j=B_j=E_j=L^{p_j}(w_j)$ due to the known properties of weighted Lebesgue spaces. 
\end{proof}

We can now give the proof of our main result:

\begin{proof}[Proof of Theorem \ref{thm:main}]
We prove the theorem in the case that the assumptions \ref{(1a)} are in force. The other cases are proved in a similar way. In particular, $T:L^{q_1}(v_1)\times L^{q_2}(v_2)\to L^q(\nu_{\vec{v},\vec{q}})$ is a bounded bilinear operator for all $\vec{q}=(q_1,q_2)$ with $q_j\in(s_j,\infty)$ $(j=1,2)$ satisfying $\frac{1}{q}=\sum_{j=1}^2\frac{1}{q_j}<1$ and all $\vec{v}=(v_1,v_2)\in A_{\vec{q}/\vec{s}}(\R^{2d})$. In addition, it is assumed that $T:L^{p_1}(u_1)\times L^{p_2}(u_2)\to L^p(\nu_{\vec{u},\vec{p}})$ is a compact operator for some $\vec{p}=(p_1,p_2)$ with $p_j\in(s_j,\infty)$ $(j=1,2)$ satisfying $\frac{1}{p}=\sum_{j=1}^2\frac{1}{p_j}<1$ and some $\vec{u}=(u_1,u_2)\in A_{\vec{p}/\vec{s}}(\R^{2d})$. We need to prove that $T:L^{r_1}(w_1)\times L^{r_2}(w_2)\to L^r(\nu_{\vec{w},\vec{r}})$ is actually compact for all $\vec{r}=(r_1,r_2)$ with $r_j\in(s_j,\infty)$ $(j=1,2)$ satisfying $\frac{1}{r}=\sum_{j=1}^2\frac{1}{r_j}<1$ and all $\vec{w}=(w_1,w_2)\in A_{\vec{r}/\vec{s}}(\R^{2d})$. Now, fix some $r_j\in(s_j,\infty)$ $(j=1,2)$ satisfying $\frac{1}{r}=\sum_{j=1}^2\frac{1}{r_j}<1$ and $\vec{w}=(w_1,w_2)\in A_{\vec{r}/\vec{s}}(\R^{2d})$. By Proposition \ref{prop:main}, we have
\begin{equation*}
  [L^{p_j}(u_j),L^{q_j}(v_j)]_\theta=L^{r_j}(w_j),\qquad [L^{p}(\nu_{\vec{u},\vec{p}}),L^{q}(\nu_{\vec{v},\vec{q}})]_\theta=L^{r}(\nu_{\vec{w},\vec{r}}),
\end{equation*}
for some $\vec{p}=(p_1,p_2)$ with $p_j\in(s_j,\infty)$ $(j=1,2)$ satisfying $\frac{1}{p}=\sum_{j=1}^2\frac{1}{p_j}<1$, some $\vec{u}=(u_1,u_2)\in A_{\vec{p}/\vec{s}}(\R^{2d})$ and some $\theta\in(0,1)$. By Lemma \ref{lem:lemOk} and by writing $A_1=L^{p_1}(u_1)$, $A_2=L^{q_1}(v_1)$, $B_1=L^{p_2}(u_2)$, $B_2=L^{q_2}(v_2)$, $E_1=L^{p}(\nu_{\vec{u},\vec{p}})$ and $E_2=L^{q}(\nu_{\vec{v},\vec{q}})$, we know that $T\in{\bf\mathcal{B}}(\bar{A}\times\bar{B},\bar{E})$,  $T:A_1^{\circ}\times B_1^{\circ}\to E_1$ is compact and that $E_1$ has also absolutely continuous norm. By Theorem \ref{thm:CCM}, it follows that $T:L^{r_1}(w_1)\times L^{r_2}(w_2)\to L^r(\nu_{\vec{w},\vec{r}})$ is also compact for all $\vec{r}=(r_1,r_2)$ with $r_j\in(s_j,\infty)$ $(j=1,2)$ satisfying $\frac{1}{r}=\sum_{j=1}^2\frac{1}{r_j}<1$ and all $\vec{w}=(w_1,w_2)\in A_{\vec{r}/\vec{s}}(\R^{2d})$.
\end{proof}

\section{Preliminaries on linear and multilinear weights}\label{section 4}

To complete the proof of Theorem \ref{thm:main}, it remains to verify Proposition \ref{prop:main}. We quote the following results which we will use in Section \ref{section 5} for the proof of Proposition \ref{prop:main}:

\begin{proposition}[\cite{RDF1985}, Theorem 1.14]\label{prop:weights} The following statement holds:
If $1<p<\infty$, we have $w\in A_p(\R^d)$ if and only if $w^{1-p'}\in
A_{p'}(\R^d)$.
\end{proposition}


\begin{theorem}[\cite{Jiao}, Theorem 2.1]\label{thm:Jiao}
Let $\vec{w}=(w_1,\dots,w_m)$, $1\leq s_j\leq p_j<\infty$ $(j=1,2,\dots,m)$ with $\frac{1}{p}=\sum_{j=1}^m\frac{1}{p_j}$ and $\frac{1}{s}=\sum_{j=1}^m\frac{1}{s_j}$. Then
$\vec{w}\in A_{\vec{p}/\vec{s}}(\R^{md})$ if and only if
\begin{equation*}
\begin{cases}
  w_j^{1-\big(\frac{p_j}{s_j}\big)^{'}}\in A_{\frac{p_js_j}{s(p_j-s_j)}}(\R^d), & p_j\neq s_j, \\
  \nu_{\vec{w},\vec{p}}\in A_{\frac{p}{s}}(\R^d),
\end{cases}
\end{equation*}
where the condition $w_j^{1-\big(\frac{p_j}{s_j}\big)^{'}}\in A_{\frac{p_js_j}{s(p_j-s_j)}}(\R^d)$ in the case $p_j=s_j$ is understood as $w_j^{s/p_j}\in A_1(\R^d)$.
\end{theorem}

\begin{remark}\label{rmk:main}
The important special case $s_1=\cdots=s_m=1$ of Theorem \ref{thm:Jiao} was already proved in \cite[Theorem 3.6]{KLOPTG}.
\end{remark}

\begin{theorem}[\cite{CWX}, Theorem 3.5, \cite{Moen}, Theorem 3.4]\label{thm:Moen} 
Let $\vec{w}=(w_1,\dots,w_m)$, $1\leq p_1,\dots,p_m<\infty$ with $\frac{1}{p}=\sum_{j=1}^m\frac{1}{p_j}$ and $p^*$ be a number $1/m\leq p\leq p^*<\infty$. Then $\vec{w}\in A_{\vec{p},p^*}(\R^{md})$ if and only if
\begin{equation*}
\begin{cases}
  w_j^{-p'_j}\in A_{mp'_j}(\R^d), & j=1,\dots,m,\\
  \nu_{\vec{w}}^{p^*}\in A_{mp^*}(\R^d),
\end{cases}
\end{equation*}
where the condition $w_j^{-p'_j}\in A_{mp'_j}(\R^d)$ in the case $p_j=1$ is understood as $w_j^{1/m}\in A_1(\R^d)$.
\end{theorem}

\begin{theorem}[\cite{BL}, Theorem 5.5.3]\label{thm:SW}
If $q_1,q_2\in[1,\infty)$ and $w_1,w_2$ are two weights, then for all $\theta\in(0,1)$ we have
\begin{equation*}
  [L^{q_1}(w_1),L^{q_2}(w_2)]_\theta=L^q(w),
\end{equation*}
where
\begin{equation}\label{eq:convexity}
  \frac{1}{q}=\frac{1-\theta}{q_1}+\frac{\theta}{q_2},\qquad
  w^{\frac{1}{q}}=w_1^{\frac{1-\theta}{q_1}}w_2^{\frac{\theta}{q_2}}.
\end{equation}
\end{theorem}

In order to present applications of Theorem \ref{thm:main} which deal with compact commutators, let us introduce relevant notation and some definitions. We will denote by $b$ a pointwise multiplier that belongs to the space
\begin{equation*}
  \BMO(\R^d):=\Big\{f:\R^d\to\C\ \Big|\ \Norm{f}{\BMO}:=\sup_Q\ave{\abs{f-\ave{f}_Q}}_Q<\infty\Big\}
\end{equation*}
 of functions of bounded mean oscillation, or its subspace
\begin{equation*}
  \CMO(\R^d):=\overline{C_c^\infty(\R^d)}^{\BMO(\R^d)},
\end{equation*}
where the closure is in the $\BMO$ norm and $C_c^{\infty}(\R^d)$ is the collection of $C^{\infty}(\R^d)$ functions with compact support.

Let $T$ denote a bilinear operator from $X_1\times X_2$ into $Y$, where $X_1,X_2$ and $Y$ are some function spaces. For $(f_1,f_2) \in X_1 \times  X_2$ and for a measurable vector $\vec{b}=(b_1,b_2)$, 
we define, whenever it makes sense, the commutators 
\begin{equation*}
\begin{split}
  [T,\vec{b}]_{e_1}(f_1,f_2)=[T,\vec{b}]_{(1,0)}(f_1,f_2) &=b_1 T(f_1,f_2)-T(b_1f_1,f_2) \\
  [T,\vec{b}]_{e_2}(f_1,f_2)=[T,\vec{b}]_{(0,1)}(f_1,f_2) &=b_2 T(f_1,f_2)-T(f_1,b_2f_2) \\
  [T,\vec{b}]_{(1,1)}(f_1,f_2) &=[[T,\vec{b}]_{e_1},\vec{b}]_{e_2}(f_1,f_2).
\end{split}
\end{equation*}
In the same way, we could define $[T,\vec{b}]_\alpha$ for any $\alpha\in\N^2$, but we will only consider the above three cases.


We also quote the following result which we need for our applications in Section \ref{section 8}:

\begin{theorem}[\cite{LMO}, Theorem 2.22]\label{thm: Ap/s bdd.}
Let $T$ be a bilinear operator and let $\vec{s}=(s_1,s_2)\in[1,\infty)^2$ with $\frac{1}{s}=\frac{1}{s_1}+\frac{1}{s_2}$. Assume that there exists $\vec{p}=(p_1,p_2)\in(s_1,\infty)\times(s_2,\infty)$, such that for all $\vec{w}=(w_1,w_2)\in A_{\vec p/\vec s}(\R^{2d})$, we have
\begin{equation*}
  \|T(f_1,f_2)\|_{L^p(\nu_{\vec{w},\vec{p}})} \lesssim \prod_{i=1}^2 \|f_i\|_{L^{p_i}(w_i)}, 
\end{equation*}
where $\frac{1}{p}=\frac{1}{p_1}+\frac{1}{p_2}$ and $\nu_{\vec{w},\vec{p}}=\prod_{i=1}^2 w_i^{p/p_i}$. Then, for all exponents $\vec{q}=(q_1,q_2)\in(s_1,\infty)\times(s_2,\infty)$, for all weights $\vec{v}=(v_1,v_2)\in A_{\vec q/\vec s}(\R^{2d})$, for all $\vec{b}=(b_1,b_2)\in\BMO(\R^d)^2$, and for each multi-index $\alpha$, we have
\begin{equation*}
  \|[T,\vec{b}]_{\alpha}(f_1,f_2)\|_{L^q(\nu_{\vec{v},\vec{q}})} \lesssim \prod_{i=1}^2\|b_i\|_{\BMO}^{\alpha_i} \|f_i\|_{L^{q_i}(v_i)},
\end{equation*}
where $\frac{1}{q}=\frac{1}{q_1}+\frac{1}{q_2}$ and $\nu_{\vec{v},\vec{q}}=\prod_{i=1}^2 v_i^{q/q_i}$.
\end{theorem}

\section{The proof of the Key Proposition \ref{prop:main}}\label{section 5}

In this section we prove Proposition \ref{prop:main}. The first step is 
to connect Theorem \ref{thm:SW} with the multilinear
$A_{\vec{p}/\vec{s}}(\R^{md})$,
$A_{\vec{p}}(\R^{md})$, and $A_{\vec{p},p^*}(\R^{md})$ conditions as follows:

\begin{lemma}\label{lem:main1}
Let
\begin{equation*}
  \vec{q}=(q_1,\dots,q_m),\quad \vec{r}=(r_1,\dots,r_m),\quad \vec{s}=(s_1,\dots,s_m)
\end{equation*}
where $s_j\in[1,\infty)$, $q_j,r_j\in(s_j,\infty)$ and
\begin{equation*}
 \frac{1}{q}=\sum_{j=1}^m\frac{1}{q_j}<1,\quad \frac{1}{r}=\sum_{j=1}^m\frac{1}{r_j}<1,\quad \frac{1}{s}=\sum_{j=1}^m\frac{1}{s_j}.
\end{equation*}
Let $\vec{v}=(v_1,\dots,v_m)\in A_{\vec{q}/\vec{s}}(\R^{md})$, $\vec{w}=(w_1,\dots,w_m)\in A_{\vec{r}/\vec{s}}(\R^{md})$. Then there exists $\vec{p}=(p_1,\dots,p_m)$, with $p_j\in(s_j,\infty)$ satisfying $\frac{1}{p}=\sum_{j=1}^m\frac{1}{p_j}<1$ and $\vec{u}=(u_1,\dots,u_m)\in A_{\vec{p}/\vec{s}}(\R^{md})$, $\theta\in(0,1)$ such that
\begin{equation}\label{maineq1}
  \frac{1}{r_j}=\frac{1-\theta}{p_j}+\frac{\theta}{q_j},\qquad w_j^{\frac{1}{r_j}}=u_j^{\frac{1-\theta}{p_j}}v_j^{\frac{\theta}{q_j}},\qquad j=1,\dots,m,
\end{equation}
and 
\begin{equation}\label{maineq2}
  \frac{1}{r}=\frac{1-\theta}{p}+\frac{\theta}{q},\qquad \nu_{\vec{w},\vec{r}}^\frac{1}{r}=\nu_{\vec{u},\vec{p}}^{\frac{1-\theta}{p}}\nu_{\vec{v},\vec{q}}^\frac{\theta}{q}.
\end{equation}
\end{lemma}

\begin{proof}
By Theorem \ref{thm:Jiao} we prove the lemma in its equivalent form: if $$v_j^{1-\big(\frac{q_j}{s_j}\big)^{'}}\in A_{\frac{s_j}{s}\big(\frac{q_j}{s_j}\big)^{'}}(\R^d), \quad \nu_{\vec{v},\vec{q}}\in A_{\frac{q}{s}}(\R^d)$$ and $$w_j^{1-\big(\frac{r_j}{s_j}\big)^{'}}\in A_{\frac{s_j}{s}\big(\frac{r_j}{s_j}\big)^{'}}(\R^d),\quad \nu_{\vec{w},\vec{r}}\in A_{\frac{r}{s}}(\R^d),$$ then there exists $\vec{p}=(p_1,\dots,p_m)$ with $p_j\in(s_j,\infty)$ with $\frac{1}{p}=\sum_{j=1}^m\frac{1}{p_j}<1$ and  $$u_j^{1-\big(\frac{p_j}{s_j}\big)^{'}}\in A_{\frac{s_j}{s}\big(\frac{p_j}{s_j}\big)^{'}}(\R^d),\quad \nu_{\vec{u},\vec{p}}\in A_{\frac{p}{s}}(\R^d),\quad\theta 
\in(0,1)$$ such that \eqref{maineq1} and \eqref{maineq2} hold.

Note that the choice of $\theta\in(0,1)$ determines 
\begin{equation*}
  p_j=p_j(\theta)=\frac{1-\theta}{\frac{1}{r_j}-\frac{\theta}{q_j}},\quad
  u_j=u_j(\theta)=w_j^{\frac{p_j}{r_j(1-\theta)}}v_j^{-\frac{p_j\cdot\theta}{q_j(1-\theta)}},\qquad j=1,\dots,m,
\end{equation*}
and
\begin{equation*}
  p=p(\theta)=\frac{1-\theta}{\frac{1}{r}-\frac{\theta}{q}},\quad
  \nu_{\vec{u},\vec{p}}=\nu_{\vec{u},\vec{p}}(\theta)=\nu_{\vec{w},\vec{r}}^{\frac{p}{r(1-\theta)}}\nu_{\vec{v},\vec{q}}^{-\frac{p\cdot\theta}{q(1-\theta)}},
\end{equation*}
so it remains to check that we can choose $\theta\in(0,1)$ so that $\vec{p}=(p_1,\dots,p_m)$ with $p_j\in(s_j,\infty)$ satisfying $\frac{1}{p}=\sum_{j=1}^m\frac{1}{p_j}<1$ and $u_j^{1-\big(\frac{p_j}{s_j}\big)^{'}}\in A_{\frac{s_j}{s}\big(\frac{p_j}{s_j}\big)^{'}}(\R^d)$, $\nu_{\vec{u},\vec{p}}\in A_{\frac{p}{s}}(\R^d)$. Since $p_j(0)=r_j\in(s_j,\infty)$ and $p(0)=r\in(1,\infty)$, the first conditions are obvious for small enough $\theta>0$ by continuity. To simplify writing, we denote $\tilde m=1/s$, $\tilde m_j=s_j/s$, $\tilde p_j=p_j/s_j$, $\tilde q_j=q_j/s_j$ and $\tilde r_j=r_j/s_j$ for $j=1,\dots,m$, and observe that these satisfy the same relations
\begin{equation*}
  \tilde p_j=\tilde p_j(\theta)=\frac{p_j(\theta)}{s_j}=\frac{1-\theta}{\frac{s_j}{r_j}-\frac{\theta s_j}{q_j}}=\frac{1-\theta}{\frac{1}{\tilde r_j}-\frac{\theta}{\tilde q_j}}
\end{equation*}
and $\tilde p_j(0)=\tilde r_j$ as the original exponents $p_j,r_j$ and $q_j$. 

We check that $u_j^{1-\tilde p'_j}\in A_{\tilde m_j\tilde p'_j}(\R^d)$, so we consider a cube $Q$ and write
\begin{equation*}
\begin{split}
  \ave{u_j^{1-\tilde p'_j}}_Q &\ave{u_j^{(1-\tilde p'_j)(-\frac{1}{\tilde m_j\tilde p'_j-1})}}_Q^{\tilde m_j\tilde p'_j-1} \\
  &=\ave{w_j^{-\frac{\tilde p'_j}{\tilde r_j(1-\theta)}}v_j^{\frac{\tilde p'_j\cdot\theta}{\tilde q_j(1-\theta)}}}_Q  
  \ave{w_j^{\frac{\tilde p'_j}{\tilde r_j(1-\theta)(\tilde m_j\tilde p'_j-1)}}v_j^{-\frac{\tilde p'_j\cdot\theta}{\tilde q_j(1-\theta)(\tilde m_j\tilde p'_j-1)}}}_Q^{\tilde m_j\tilde p'_j-1}.
\end{split}
\end{equation*}

In the first average, we use H\"older's inequality with exponents $1+\eps^{\pm 1}$, and in the second with exponents $1+\delta^{\pm 1}$ to get
\begin{equation}\label{eq:beforeRHI1}
\begin{split}
  &\leq \ave{w_j^{-\frac{\tilde p'_j(1+\eps)}{\tilde r_j(1-\theta)}}}_Q^{\frac{1}{1+\eps}} \ave{v_j^{\frac{\tilde p'_j\cdot\theta(1+\eps)}{\tilde q_j\eps(1-\theta)}}}_Q^{\frac{\eps}{1+\eps}} \\ &\qquad\times\ave{w_j^{\frac{\tilde p'_j(1+\delta)}{\tilde r_j(1-\theta)(\tilde m_j\tilde p'_j-1)}}}_Q^{\frac{\tilde m_j\tilde p'_j-1}{1+\delta}} \ave{v_j^{-\frac{\tilde p'_j\cdot\theta(1+\delta)}{\tilde q_j\delta(1-\theta)(\tilde m_j\tilde p'_j-1)}}}_Q^{\frac{(\tilde m_j\tilde p'_j-1)\delta}{1+\delta}}  \\
  &=\ave{(w_j^{1-\tilde r'_j})^{\tilde\varrho_j(\theta)}}_Q^{\frac{1}{1+\eps}}\ave{(v_j^{-\frac{1-\tilde q'_j}{\tilde m_j\tilde q'_j-1}})^{\tilde\sigma_j(\theta)}}_Q^{\frac{\eps}{1+\eps}}  \\
  &\qquad\times\ave{(w_j^{-\frac{1-\tilde r'_j}{\tilde m_j\tilde r'_j-1}})^{\tilde\tau_j(\theta)}}_Q^{\frac{\tilde m_j\tilde p'_j-1}{1+\delta}}\ave{(v_j^{1-\tilde q'_j})^{\tilde\phi_j(\theta)}}_Q^{\frac{(\tilde m_j\tilde p'_j-1)\delta}{1+\delta}},
\end{split}  
\end{equation}
where
\begin{equation*}
  \tilde\varrho_j(\theta):=\frac{\tilde p'_j(\theta)(1+\eps)}{\tilde r'_j(1-\theta)},\qquad
  \tilde\sigma_j(\theta):=\frac{\theta \tilde p'_j(\theta)(\tilde m_j\tilde q'_j-1)(1+\eps)}{\tilde q'_j\eps(1-\theta)},
\end{equation*}
and 
\begin{equation*}
  \tilde\tau_j(\theta):=\frac{\tilde p'_j(\theta)(\tilde m_j \tilde r'_j-1)(1+\delta)}{\tilde r'_j(1-\theta)(\tilde m_j \tilde p'_j(\theta)-1)},\qquad
  \tilde\phi_j(\theta):=\frac{\theta \tilde p'_j(\theta)(1+\delta)}{\tilde q'_j\delta(1-\theta)(\tilde m_j\tilde p'_j(\theta)-1)}.
\end{equation*}

Now, we choose $\eps=\eps(\theta)$ and $\delta=\delta(\theta)$ in such a way that
\begin{equation*}
 \tilde\varrho_j(\theta)=\tilde\sigma_j(\theta),\quad \tilde\tau_j(\theta)=\tilde\phi_j(\theta),
\end{equation*}
which has the solution
\begin{equation*}
  \eps(\theta)=\frac{\theta \tilde r'_j(\tilde m_j\tilde q'_j-1)}{\tilde q'_j},\quad
  \delta(\theta)=\frac{\theta \tilde r'_j}{\tilde q'_j(\tilde m_j\tilde r'_j-1)}.
\end{equation*}

The strategy to proceed is to use the reverse H\"older inequality for $A_v(\R^d)$ weights due to Coifman--Fefferman \cite{CF}, which says that each $W\in A_v(\R^d)$ satisfies
\begin{equation}\label{eq:RHI}
  \ave{W^t}_Q^{1/t}\lesssim \ave{W}_Q
\end{equation}
for all $t\leq 1+\eta$ and for some $\eta>0$ depending only on $[W]_{A_v}$.

Recalling that $\tilde p_j(0)=\tilde r_j$, we see that $\tilde\varrho_j(0)=\tilde\tau_j(0)=1$. By continuity, given any $\eta>0$, we find that 
\begin{equation*}
  \max(\tilde\varrho_j(\theta),\tilde\tau_j(\theta))\leq 1+\eta\,\,\,\text{for all small enough}\,\,\,\theta>0.
\end{equation*}
By Proposition \ref{prop:weights} each of the four functions 
\begin{equation*}
\begin{split}
  w_j^{1-\tilde r'_j}\in A_{\tilde m_j\tilde r'_j}(\R^d),\quad 
  & w_j^{-\frac{1-\tilde r'_j}{\tilde m_j\tilde r'_j-1}}\in A_{(\tilde m_j\tilde r'_j)'}(\R^d),\\
  & v_j^{1-\tilde q'_j}\in A_{\tilde m_j\tilde q'_j}(\R^d),\quad 
  v_j^{-\frac{1-\tilde q'_j}{\tilde m_j\tilde q'_j-1}}\in A_{(\tilde m_j\tilde q'_j)'}(\R^d)
\end{split}
\end{equation*}
satisfies the reverse H\"older inequality (\ref{eq:RHI}) for all $t\leq 1+\eta$ and for some $\eta>0$. Thus, for all small enough $\theta>0$, we have
\begin{equation*}
\begin{split}
  \eqref{eq:beforeRHI1}
  &\lesssim\ave{w_j^{1-\tilde r'_j}}_Q^{\frac{\tilde p'_j}{\tilde r'_j(1-\theta)}}\ave{v_j^{-\frac{1-\tilde q'_j}{\tilde m_j \tilde q'_j-1}}}_Q^{\frac{\theta \tilde p'_j(\tilde m_j \tilde q'_j-1)}{\tilde q'_j(1-\theta)}} \\
  &\qquad\times\ave{w_j^{-\frac{1-\tilde r'_j}{\tilde m_j\tilde r'_j-1}}}_Q^{\frac{\tilde p'_j(\tilde m_j\tilde r'_j-1)}{\tilde r'_j(1-\theta)}}\ave{v_j^{1-\tilde q'_j}}_Q^{\frac{\theta \tilde p'_j}{\tilde q'_j(1-\theta)}}  \\
  &=(\ave{w_j^{1-\tilde r'_j}}_Q\ave{w_j^{-\frac{1-\tilde r'_j}{\tilde m_j\tilde r'_j-1}}}_Q^{\tilde m_j\tilde r'_j-1})^{\frac{\tilde p'_j}{\tilde r'_j(1-\theta)}}  \\
  &\qquad\times(\ave{v_j^{1-\tilde q'_j}}_Q\ave{v_j^{-\frac{1-\tilde q'_j}{\tilde m_j\tilde q'_j-1}}}^{\tilde m_j\tilde q'_j-1}_Q)^{\frac{\theta \tilde p'_j}{\tilde q'_j(1-\theta)}}  \\
  &\leq[w_j^{1-\tilde r'_j}]_{A_{\tilde m_j\tilde r'_j}}^{\frac{\tilde q'_j}{\tilde q'_j-\theta \tilde r'_j}}[v_j^{1-\tilde q'_j}]_{A_{\tilde m_j\tilde q'_j}}^{\frac{\theta \tilde r_j(\tilde q_j-1)}{\tilde q_j-\theta \tilde r_j}}.
\end{split}
\end{equation*}
In combination with the lines preceding \eqref{eq:beforeRHI1}, we have shown that
\begin{equation*}
  [u_j^{1-\tilde p'_j}]_{A_{\tilde m_j\tilde p'_j}}\lesssim[w_j^{1-\tilde r'_j}]_{A_{\tilde m_j\tilde r'_j}}^{\frac{\tilde q'_j}{\tilde q'_j-\theta \tilde r'_j}}[v_j^{1-\tilde q'_j}]_{A_{\tilde m_j\tilde q'_j}}^{\frac{\theta \tilde r_j(\tilde q_j-1)}{\tilde q_j-\theta r_j}}<\infty,
\end{equation*}
provided that $\theta>0$ is small enough. 

Now, we check that $\nu_{\vec{u},\vec{p}}\in A_{\tilde mp}(\R^d)$, so we consider a cube $Q$ and write
\begin{equation*}
  \ave{\nu_{\vec{u},\vec{p}}}_Q\ave{\nu_{\vec{u},\vec{p}}^{-\frac{1}{\tilde mp-1}}}_Q^{\tilde mp-1}
  =\ave{\nu_{\vec{w},\vec{r}}^{\frac{p}{r(1-\theta)}}\nu_{\vec{v},\vec{q}}^{-\frac{p\cdot\theta}{q(1-\theta)}}}_Q  \ave{\nu_{\vec{w},\vec{r}}^{-\frac{p}{r(1-\theta)(\tilde mp-1)}}\nu_{\vec{v},\vec{q}}^{\frac{p\cdot\theta}{q(1-\theta)(\tilde mp-1)}}}_Q^{\tilde mp-1}.
\end{equation*}

In the first average, we use H\"older's inequality with exponents $1+\eps^{\pm 1}$, and in the second with exponents $1+\delta^{\pm 1}$ to get
\begin{equation}\label{eq:beforeRHI2}
\begin{split}
  &\leq \ave{\nu_{\vec{w},\vec{r}}^{\frac{p(1+\eps)}{r(1-\theta)}}}_Q^{\frac{1}{1+\eps}} \ave{\nu_{\vec{v},\vec{q}}^{-\frac{p\cdot\theta(1+\eps)}{q\eps(1-\theta)}}}_Q^{\frac{\eps}{1+\eps}}  
  \ave{\nu_{\vec{w},\vec{r}}^{-\frac{p(1+\delta)}{r(1-\theta)(\tilde mp-1)}}}_Q^{\frac{\tilde mp-1}{1+\delta}} \ave{\nu_{\vec{v},\vec{q}}^{\frac{p\cdot\theta(1+\delta)}{q\delta(1-\theta)(\tilde mp-1)}}}_Q^{\frac{(\tilde mp-1)\delta}{1+\delta}}  \\
  &=\ave{\nu_{\vec{w},\vec{r}}^{\tilde \varrho(\theta)}}_Q^{\frac{1}{1+\eps}}\ave{(\nu_{\vec{v},\vec{q}}^{-\frac{1}{\tilde mq-1}})^{\tilde \sigma(\theta)}}_Q^{\frac{\eps}{1+\eps}} \ave{(\nu_{\vec{w},\vec{r}}^{-\frac{1}{\tilde mr-1}})^{\tilde \tau(\theta)}}_Q^{\frac{\tilde mp-1}{1+\delta}}\ave{\nu_{\vec{v},\vec{q}}^{\tilde \phi(\theta)}}_Q^{\frac{(\tilde mp-1)\delta}{1+\delta}},
\end{split}  
\end{equation}
where
\begin{equation*}
  \tilde \varrho(\theta):=\frac{p(\theta)(1+\eps)}{r(1-\theta)},\qquad
  \tilde \sigma(\theta):=\frac{\theta p(\theta)(\tilde mq-1)(1+\eps)}{q\eps(1-\theta)},
\end{equation*}
and
\begin{equation*}
  \tilde \tau(\theta):=\frac{p(\theta)(\tilde mr-1)(1+\delta)}{r(1-\theta)(\tilde mp(\theta)-1)},\qquad
  \tilde \phi(\theta):=\frac{\theta p(\theta)(1+\delta)}{q\delta(1-\theta)(\tilde mp(\theta)-1)}.
\end{equation*}

Again, we choose $\eps=\eps(\theta)$ and $\delta=\delta(\theta)$  in such a way that
\begin{equation*}
   \tilde \varrho(\theta)=\tilde \sigma(\theta),\quad \tilde \tau(\theta)=\tilde \phi(\theta),
\end{equation*}
which gives
\begin{equation*}
  \eps(\theta)=\theta r(\tilde m-\frac{1}{q}),\quad
  \delta(\theta)=\frac{\theta r}{q(\tilde mr-1)}.
\end{equation*}
%

The strategy to proceed is the same as before. In particular, we use the reverse H\"older inequality \eqref{eq:RHI} for $A_v(\R^d)$ weights. 

Recalling that $p(0)=r$, we see that $\tilde\varrho(0)=\tilde\tau(0)=1$. By continuity, given any $\eta>0$, we find that 
\begin{equation*}
  \max(\tilde \varrho(\theta),\tilde \tau(\theta))\leq 1+\eta\,\,\,\text{for all small enough}\,\,\,\theta>0.
\end{equation*}
By Proposition \ref{prop:weights} each of the four functions
\begin{equation*}
\begin{split}
   \nu_{\vec{w},\vec{r}}\in A_{\tilde mr}(\R^d),\quad
   &\nu_{\vec{w},\vec{r}}^{-\frac{1}{\tilde mr-1}}\in A_{(\tilde mr)'}(\R^d),\\
   &\nu_{\vec{v},\vec{q}}\in A_{\tilde mq}(\R^d),\quad \nu_{\vec{v},\vec{q}}^{-\frac{1}{\tilde mq-1}}\in A_{(\tilde mq)'}(\R^d)
\end{split}
\end{equation*}
satisfies the reverse H\"older inequality (\ref{eq:RHI}) for all $t\leq 1+\eta$ and for some $\eta>0$. Thus, for all small enough $\theta>0$, we have
\begin{equation*}
\begin{split}
  \eqref{eq:beforeRHI2}
  &\lesssim\ave{\nu_{\vec{w},\vec{r}}}_Q^{\frac{p(\theta)}{r(1-\theta)}}\ave{\nu_{\vec{v},\vec{q}}^{-\frac{1}{\tilde mq-1}}}_Q^{\frac{\theta p(\theta)(\tilde mq-1)}{q(1-\theta)}}  \\
  &\qquad\times\ave{\nu_{\vec{w},\vec{r}}^{-\frac{1}{\tilde mr-1}}}_Q^{\frac{p(\theta)(\tilde mr-1)}{r(1-\theta)}}\ave{\nu_{\vec{v},\vec{q}}}_Q^{\frac{\theta p(\theta)}{q(1-\theta)}}  \\
  &=(\ave{\nu_{\vec{w},\vec{r}}}_Q\ave{\nu_{\vec{w},\vec{r}}^{-\frac{1}{\tilde mr-1}}}_Q^{\tilde mr-1})^{\frac{p(\theta)}{r(1-\theta)}}  \\
  &\qquad\times(\ave{\nu_{\vec{v},\vec{q}}}_Q\ave{\nu_{\vec{v},\vec{q}}^{-\frac{1}{\tilde mq-1}}}^{\tilde mq-1}_Q)^{\frac{\theta p(\theta)}{q(1-\theta)}}  \\
  &\leq[\nu_{\vec{w},\vec{r}}]_{A_{\tilde mr}}^{\frac{q}{q-\theta r}}[\nu_{\vec{v},\vec{q}}]_{A_{\tilde mq}}^{\frac{\theta r}{q-\theta r}}.
\end{split}
\end{equation*}
In combination with the lines preceding \eqref{eq:beforeRHI2}, we have shown that
\begin{equation*}
  [\nu_{\vec{u},\vec{p}}]_{A_{\tilde mp}}\lesssim[\nu_{\vec{w},\vec{r}}]_{A_{\tilde mr}}^{\frac{q}{q-\theta r}}[\nu_{\vec{v},\vec{q}}]_{A_{\tilde mq}}^{\frac{\theta r}{q-\theta r}}<\infty,
\end{equation*}
provided that $\theta>0$ is small enough. This concludes the proof.
\end{proof}

\begin{lemma}\label{lem:main3}
Let 
\begin{equation*}
  \alpha\ge 0,\quad \vec{q}=(q_1,\dots,q_m),\quad \vec{r}=(r_1,\dots,r_m)
\end{equation*}
where $1<q_1,\dots,q_m<\infty$, $1<r_1,\dots,r_m<\infty$ and
\begin{equation*}
  \frac{1}{q}=\sum_{j=1}^m\frac{1}{q_j}\in(\alpha,\alpha+1), 
  \frac{1}{q^*}=\frac{1}{q}-\alpha,
  \frac{1}{r}=\sum_{j=1}^m\frac{1}{r_j}\in(\alpha,\alpha+1), 
  \frac{1}{r^*}=\frac{1}{r}-\alpha.
\end{equation*}
Let $\vec{w}=(w_1,\dots, w_m)\in A_{\vec{r},r^*}(\R^{md})$, $\vec{v}=(v_1,\dots,v_m)\in A_{\vec{q},q^*}(\R^{md})$. Then there exists $\vec{p}=(p_1,\dots,p_m)$ with $1<p_1,\dots,p_m<\infty$ satisfying $\frac{1}{p}=\sum_{j=1}^m\frac{1}{p_j}\in(\alpha,\alpha+1)$, $\frac{1}{p^*}=\frac{1}{p}-\alpha$, $\frac{1}{q}-\frac{1}{q^*}=\frac{1}{p}-\frac{1}{p^*}=\frac{1}{r}-\frac{1}{r^*}$, $\vec{u}=(u_1,\dots,u_m)\in A_{\vec{p},p^*}(\R^{md})$, $\theta\in(0,1)$ such that
\begin{equation}\label{maineq3}
  \frac{1}{r_j}=\frac{1-\theta}{p_j}+\frac{\theta}{q_j},\qquad w_j=u_j^{1-\theta}v_j^{\theta},\qquad j=1,\dots,m,
\end{equation}
and 
\begin{equation}\label{maineq4}
  \frac{1}{r}=\frac{1-\theta}{p}+\frac{\theta}{q},\qquad \nu_{\vec{w}}=\nu_{\vec{u}}^{1-\theta}\nu_{\vec{v}}^{\theta}.
\end{equation}
\end{lemma}

\begin{proof}
Using Theorem \ref{thm:Moen}, we prove the lemma in its equivalent form: if $$v_j^{-q'_j}\in A_{mq'_j}(\R^d), \quad
\nu_{\vec{v}}^{q^*}\in A_{mq^*}(\R^d)$$ and 
$$w_j^{-r'_j}\in A_{mr'_j}(\R^d), \quad  
\nu_{\vec{w}}^{r^*}\in A_{mr^*}(\R^d),$$ then there exists $\vec{p}=(p_1,\dots,p_m)$ with $1<p_1,\dots,p_m<\infty$ with  $\frac{1}{p}=\sum_{j=1}^m\frac{1}{p_j}\in(\alpha,\alpha+1)$, $\frac{1}{p^*}=\frac{1}{p}-\alpha$ , $\frac{1}{q}-\frac{1}{q^*}=\frac{1}{p}-\frac{1}{p^*}=\frac{1}{r}-\frac{1}{r^*}$ and 
$$u_j^{-p'_j}\in A_{mp'_j}(\R^d), \quad
\nu_{\vec{u}}^{p^*}\in A_{mp^*}(\R^d), \quad\theta 
\in(0,1)$$ such that \eqref{maineq3} and \eqref{maineq4} hold.

Note that the choice of $\theta\in(0,1)$ determines 
\begin{equation*}
  p_j=p_j(\theta)=\frac{1-\theta}{\frac{1}{r_j}-\frac{\theta}{q_j}},\qquad
  u_j=u_j(\theta)=w_j^{\frac{1}{1-\theta}}v_j^{-\frac{\theta}{1-\theta}},\qquad j=1,\dots,m,
\end{equation*}
and
\begin{equation*}
  p=p(\theta)=\frac{1-\theta}{\frac{1}{r}-\frac{\theta}{q}},\qquad
  \nu_{\vec{u}}=\nu_{\vec{u}}(\theta)=\nu_{\vec{w}}^{\frac{1}{1-\theta}}\nu_{\vec{v}}^{-\frac{\theta}{1-\theta}},
\end{equation*}
so it remains to check that we can choose $\theta\in(0,1)$ so that $\vec{p}=(p_1,\dots,p_m)$ with $1<p_1,\dots,p_m<\infty$ satisfying $\frac{1}{p}=\sum_{j=1}^m\frac{1}{p_j}\in(\alpha,\alpha+1)$, $\frac{1}{p^*}=\frac{1}{p}-\alpha$, $\frac{1}{q}-\frac{1}{q^*}=\frac{1}{p}-\frac{1}{p^*}=\frac{1}{r}-\frac{1}{r^*}$ and $u_j^{-p'_j}\in A_{mp'_j}(\R^d)$, $\nu_{\vec{u}}^{p^*}\in A_{mp^*}(\R^d)$. Since $1<p_j(0)=r_j<\infty$ and $1/(\alpha+1)<p(0)=r<1/\alpha$, the first conditions are obvious for small enough $\theta>0$ by continuity. 

We check that $u_j^{-p'_j}\in A_{mp'_j}(\R^d)$, so we consider a cube $Q$ and write
\begin{equation*}
\begin{split}
  \ave{u_j^{-p'_j}}_Q\ave{u_j^{(-p'_j)(-\frac{1}{mp'_j-1})}}_Q^{mp'_j-1}
  &=\ave{w_j^{-\frac{p'_j}{1-\theta}}v_j^{\frac{p'_j\cdot\theta}{1-\theta}}}_Q  \\
  &\qquad\times\ave{w_j^{\frac{p'_j}{(1-\theta)(mp'_j-1)}}v_j^{-\frac{p'_j\cdot\theta}{(1-\theta)(mp'_j-1)}}}_Q^{mp'_j-1}.
\end{split}
\end{equation*}

In the first average, we use H\"older's inequality with exponents $1+\eps^{\pm 1}$, and in the second with exponents $1+\delta^{\pm 1}$ to get
\begin{equation}\label{eq:beforeRHI3}
\begin{split}
  &\leq \ave{w_j^{-\frac{p'_j(1+\eps)}{1-\theta}}}_Q^{\frac{1}{1+\eps}} \ave{v_j^{\frac{p'_j\cdot\theta(1+\eps)}{\eps(1-\theta)}}}_Q^{\frac{\eps}{1+\eps}} \\ &\qquad\times\ave{w_j^{\frac{p'_j(1+\delta)}{(1-\theta)(mp'_j-1)}}}_Q^{\frac{mp'_j-1}{1+\delta}} \ave{v_j^{-\frac{p'_j\cdot\theta(1+\delta)}{\delta(1-\theta)(mp'_j-1)}}}_Q^{\frac{(mp'_j-1)\delta}{1+\delta}}  \\
  &=\ave{(w_j^{-r'_j})^{\varrho_j(\theta)}}_Q^{\frac{1}{1+\eps}}\ave{(v_j^{\frac{q'_j}{mq'_j-1}})^{\sigma_j(\theta)}}_Q^{\frac{\eps}{1+\eps}}  \\
  &\qquad\times\ave{(w_j^{\frac{r'_j}{mr'_j-1}})^{\tau_j(\theta)}}_Q^{\frac{mp'_j-1}{1+\delta}}\ave{(v_j^{-q'_j})^{\phi_j(\theta)}}_Q^{\frac{(mp'_j-1)\delta}{1+\delta}},
\end{split}  
\end{equation}
where
\begin{equation*}
  \varrho_j(\theta):=\frac{p'_j(\theta)(1+\eps)}{r'_j(1-\theta)},\quad
  \sigma_j(\theta):=\frac{\theta p'_j(\theta)(mq'_j-1)(1+\eps)}{q'_j\eps(1-\theta)},
\end{equation*}
and 
\begin{equation*}
  \tau_j(\theta):=\frac{p'_j(\theta)(mr'_j-1)(1+\delta)}{r'_j(1-\theta)(mp'_j(\theta)-1)},\quad
  \phi_j(\theta):=\frac{\theta p'_j(\theta)(1+\delta)}{q'_j\delta(1-\theta)(mp'_j(\theta)-1)}.
\end{equation*}

As in the proof of Lemma \ref{lem:main1}, we choose $\eps=\eps(\theta)$ and $\delta=\delta(\theta)$ in such a way that
\begin{equation*}
  \varrho_j(\theta)=\sigma_j(\theta),\quad \tau_j(\theta)=\phi_j(\theta),
\end{equation*}
which is the same as
\begin{equation*}
  \eps(\theta)=\frac{\theta r'_j(mq'_j-1)}{q'_j},\quad \delta(\theta) =\frac{\theta r'_j}{q'_j(mr'_j-1)}.
\end{equation*}


The strategy to proceed is also the same as in the proof of Lemma \ref{lem:main1}. In particular, we use the reverse H\"older inequality \eqref{eq:RHI} for $A_v(\R^d)$ weights. 

Recalling that $p_j(0)=r_j$, we see that $\varrho_j(0)=\tau_j(0)=1$. By continuity, given any $\eta>0$, we find that 
\begin{equation*}
  \max(\varrho_j(\theta),\tau_j(\theta))\leq 1+\eta\,\,\,\text{for all small enough}\,\,\,\theta>0.
\end{equation*}
By Proposition \ref{prop:weights} each of the four functions
\begin{equation*}
\begin{split}
  w_j^{-r'_j}\in A_{mr'_j}(\R^d), \quad 
  & w_j^{\frac{r'_j}{mr'_j-1}}\in A_{(mr'_j)'}(\R^d), \\
  & v_j^{-q'_j}\in A_{mq'_j}(\R^d), \quad
  v_j^{\frac{q'_j}{mq'_j-1}}\in A_{(mq'_j)'}(\R^d)
\end{split}
\end{equation*}
satisfies the reverse H\"older inequality (\ref{eq:RHI}) for all $t\leq 1+\eta$ and for some $\eta>0$. Thus, for all small enough $\theta>0$, we have
\begin{equation*}
\begin{split}
  \eqref{eq:beforeRHI3}
  &\lesssim\ave{w_j^{-r'_j}}_Q^{\frac{p'_j}{r'_j(1-\theta)}}\ave{v_j^{\frac{q'_j}{mq'_j-1}}}_Q^{\frac{\theta p'_j(mq'_j-1)}{q'_j(1-\theta)}} \\
  &\qquad\times\ave{w_j^{\frac{r'_j}{mr'_j-1}}}_Q^{\frac{p'_j(mr'_j-1)}{r'_j(1-\theta)}}\ave{v_j^{-q'_j}}_Q^{\frac{\theta p'_j}{q'_j(1-\theta)}}  \\
  &=(\ave{w_j^{-r'_j}}_Q\ave{w_j^{\frac{r'_j}{mr'_j-1}}}_Q^{mr'_j-1})^{\frac{p'_j}{r'_j(1-\theta)}}  \\
  &\qquad\times(\ave{v_j^{-q'_j}}_Q\ave{v_j^{\frac{q'_j}{mq'_j-1}}}^{mq'_j-1}_Q)^{\frac{\theta p'_j}{q'_j(1-\theta)}}  \\
  &\leq[w_j^{-r'_j}]_{A_{mr'_j}}^{\frac{q'_j}{q'_j-\theta r'_j}}[v_j^{-q'_j}]_{A_{mq'_j}}^{\frac{\theta r_j(q_j-1)}{q_j-\theta r_j}}.
\end{split}
\end{equation*}
In combination with the lines preceding \eqref{eq:beforeRHI3}, we have shown that
\begin{equation*}
  [u_j^{1-p'_j}]_{A_{mp'_j}}\lesssim[w_j^{-r'_j}]_{A_{mr'_j}}^{\frac{q'_j}{q'_j-\theta r'_j}}[v_j^{-q'_j}]_{A_{mq'_j}}^{\frac{\theta r_j(q_j-1)}{q_j-\theta r_j}}<\infty,
\end{equation*}
provided that $\theta>0$ is small enough. 

Now, we check that $\nu_{\vec{u}}^{p^*}\in A_{mp^*}(\R^d)$, so we consider a cube $Q$ and write
\begin{equation*}
  \ave{\nu_{\vec{u}}^{p^*}}_Q\ave{\nu_{\vec{u}}^{-\frac{p^*}{mp^*-1}}}_Q^{mp^*-1}
  =\ave{\nu_{\vec{w}}^{\frac{p^*}{1-\theta}}\nu_{\vec{v}}^{-\frac{p^*\cdot\theta}{1-\theta}}}_Q  \ave{\nu_{\vec{w}}^{-\frac{p^*}{(1-\theta)(mp^*-1)}}\nu_{\vec{v}}^{\frac{p^*\cdot\theta}{(1-\theta)(mp^*-1)}}}_Q^{mp^*-1}.
\end{equation*}

In the first average, we use H\"older's inequality with exponents $1+\eps^{\pm 1}$, and in the second with exponents $1+\delta^{\pm 1}$ to get
\begin{equation}\label{eq:beforeRHI4}
\begin{split}
  &\leq \ave{\nu_{\vec{w}}^{\frac{p^*(1+\eps)}{1-\theta}}}_Q^{\frac{1}{1+\eps}} \ave{\nu_{\vec{v}}^{-\frac{p^*\cdot\theta(1+\eps)}{\eps(1-\theta)}}}_Q^{\frac{\eps}{1+\eps}}  
  \ave{\nu_{\vec{w}}^{-\frac{p^*(1+\delta)}{(1-\theta)(mp^*-1)}}}_Q^{\frac{mp^*-1}{1+\delta}} \ave{\nu_{\vec{v}}^{\frac{p^*\cdot\theta(1+\delta)}{\delta(1-\theta)(mp^*-1)}}}_Q^{\frac{(mp^*-1)\delta}{1+\delta}}  \\
  &=\ave{({\nu_{\vec{w}}^{r^*}})^{\varrho(\theta)}}_Q^{\frac{1}{1+\eps}}\ave{(\nu_{\vec{v}}^{-\frac{q^*}{mq^*-1}})^{\sigma(\theta)}}_Q^{\frac{\eps}{1+\eps}} \ave{(\nu_{\vec{w}}^{-\frac{r^*}{mr^*-1}})^{\tau(\theta)}}_Q^{\frac{mp^*-1}{1+\delta}}\ave{(\nu_{\vec{v}}^{q^*})^{\phi(\theta)}}_Q^{\frac{(mp^*-1)\delta}{1+\delta}},
\end{split}  
\end{equation}
where
\begin{equation*}
  \varrho(\theta):=\frac{p^*(1+\eps)}{r^*(1-\theta)},\qquad
  \sigma(\theta):=\frac{\theta p^*(mq^*-1)(1+\eps)}{q^*\eps(1-\theta)},
\end{equation*}
and
\begin{equation*}
  \tau(\theta):=\frac{p^*(mr^*-1)(1+\delta)}{r^*(1-\theta)(mp^*-1)},\qquad
  \phi(\theta):=\frac{\theta p^*(1+\delta)}{q^*\delta(1-\theta)(mp^*-1)}.
\end{equation*}

Again, we choose $\eps=\eps(\theta)$ and $\delta=\delta(\theta)$ in such a way that 
\begin{equation*}
  \varrho(\theta)=\sigma(\theta),\quad  \tau(\theta)=\phi(\theta),
\end{equation*}
which means that
\begin{equation*}
  \eps(\theta)=\frac{\theta r^*(mq^*-1)}{q^*},\quad \delta(\theta)=\frac{\theta r^*}{q^*(mr^*-1)}.
\end{equation*}


The strategy to proceed is the same as before. In particular, we use the reverse H\"older inequality \eqref{eq:RHI} for $A_v(\R^d)$ weights. 

Recalling that $p(0)=r$, we see that $\varrho(0)=\tau(0)=1$. By continuity, given any $\eta>0$, we find that 
\begin{equation*}
  \max(\varrho(\theta),\tau(\theta))\leq 1+\eta\,\,\,\text{for all small enough}\,\,\,\theta>0.
\end{equation*}
By Proposition \ref{prop:weights} each of the four functions
\begin{equation*}
\begin{split}
  \nu_{\vec{w}}^{r^*}\in A_{mr^*}(\R^d), \quad  &\nu_{\vec{w}}^{-\frac{r^*}{mr^*-1}}\in A_{(mr^*)'}(\R^d), \\ &\nu_{\vec{v}}^{q^*}\in A_{mq^*}(\R^d), \quad  \nu_{\vec{v}}^{-\frac{q^*}{mq^*-1}}\in A_{(mq^*)'}(\R^d) 
\end{split}
\end{equation*}
satisfies the reverse H\"older inequality (\ref{eq:RHI}) for all $t\leq 1+\eta$ and for some $\eta>0$. Thus, for all small enough $\theta>0$, we have
\begin{equation*}
\begin{split}
  \eqref{eq:beforeRHI4}
  &\lesssim\ave{\nu_{\vec{w}}^{r^*}}_Q^{\frac{p^*}{r^*(1-\theta)}}\ave{\nu_{\vec{v}}^{-\frac{q^*}{mq^*-1}}}_Q^{\frac{\theta p^*(mq^*-1)}{q^*(1-\theta)}}  \\
  &\qquad\times\ave{\nu_{\vec{w}}^{-\frac{r^*}{mr^*-1}}}_Q^{\frac{p^*(mr^*-1)}{r^*(1-\theta)}}\ave{\nu_{\vec{v}}^{q^*}}_Q^{\frac{\theta p^*}{q^*(1-\theta)}}  \\
  &=(\ave{\nu_{\vec{w}}^{r^*}}_Q\ave{\nu_{\vec{w}}^{-\frac{r^*}{mr^*-1}}}_Q^{mr^*-1})^{\frac{p^*}{r^*(1-\theta)}}  \\
  &\qquad\times(\ave{\nu_{\vec{v}}^{q^*}}_Q\ave{\nu_{\vec{v}}^{-\frac{q^*}{mq^*-1}}}^{mq^*-1}_Q)^{\frac{\theta p^*}{q^*(1-\theta)}}  \\
  &\leq[\nu_{\vec{w}}^{r^*}]_{A_{mr^*}}^{\frac{p^*}{r^*(1-\theta)}}[\nu_{\vec{v}}^{q^*}]_{A_{mq^*}}^{\frac{\theta p^*}{q^*(1-\theta)}}.
\end{split}
\end{equation*}
In combination with the lines preceding \eqref{eq:beforeRHI4}, we have shown that
\begin{equation*}
  [\nu_{\vec{u}}^{p^*}]_{A_{mp^*}}\lesssim[\nu_{\vec{w}}^{r^*}]_{A_{mr^*}}^{\frac{p^*}{r^*(1-\theta)}}[\nu_{\vec{v}}^{q^*}]_{A_{mq^*}}^{\frac{\theta p^*}{q^*(1-\theta)}}<\infty,
\end{equation*}
provided that $\theta>0$ is small enough. This concludes the proof.
\end{proof}

We can also connect Theorem \ref{thm:SW} with the linear $A_{p_j/s_j}(\R^d)$, $A_{p_j}(\R^d)$ and $A_{p_j,p_j^*}(\R^d)$ conditions as follows:

\begin{lemma}\label{lem:main4}
Let 
\begin{equation*}
  \vec{q}=(q_1,\dots,q_m),\quad \vec{r}=(r_1,\dots,r_m),\quad
  \vec{s}=(s_1,\dots,s_m) 
\end{equation*}
where $s_j\in[1,\infty)$, $q_j,r_j\in(s_j,\infty)$ and
\begin{equation*}
  \frac{1}{q}=\sum_{j=1}^m\frac{1}{q_j}<1, \quad \frac{1}{r}=\sum_{j=1}^m\frac{1}{r_j}<1, \quad \frac{1}{s}=\sum_{j=1}^m\frac{1}{s_j}. 
\end{equation*}
Let $\vec{v}=(v_1,\dots,v_m)\in \prod_{j=1}^m A_{q_j/s_j}(\R^{d})$, $\vec{w}=(w_1,\dots,w_m)\in\prod_{j=1}^m A_{r_j/s_j}(\\\R^{d})$. Then there exists $\vec{p}=(p_1,\dots,p_m)$, with $p_j\in(s_j,\infty)$ satisfying $\frac{1}{p}=\sum_{j=1}^m\frac{1}{p_j}<1$ and $\vec{u}=(u_1,\dots,u_m)\in\prod_{j=1}^m A_{p_j/s_j}(\R^{d})$, $\theta\in(0,1)$ such that
\begin{equation*}
  \frac{1}{r_j}=\frac{1-\theta}{p_j}+\frac{\theta}{q_j},\qquad w_j^{\frac{1}{r_j}}=u_j^{\frac{1-\theta}{p_j}}v_j^{\frac{\theta}{q_j}},\qquad j=1,\dots,m,
\end{equation*}
and 
\begin{equation*}
  \frac{1}{r}=\frac{1-\theta}{p}+\frac{\theta}{q},\qquad \nu_{\vec{w},\vec{r}}^\frac{1}{r}=\nu_{\vec{u},\vec{p}}^{\frac{1-\theta}{p}}\nu_{\vec{v},\vec{q}}^\frac{\theta}{q}.
\end{equation*}
\end{lemma}

\begin{proof}
This follows by applying \cite[Lemma 4.4]{HL} to each component separately. 
\end{proof}

%

\begin{lemma}\label{lem:main6}
Let 
\begin{equation*}
  \alpha\ge 0,\quad \vec{q}=(q_1,\dots,q_m),\quad \vec{r}=(r_1,\dots,r_m)
\end{equation*}
where $1<q_1,\dots,q_m<\infty$, $1<r_1,\dots,r_m<\infty$ and
\begin{equation*}
  \frac{1}{q}=\sum_{j=1}^m\frac{1}{q_j}\in(\alpha,\alpha+1), \frac{1}{\tilde q_j}=\frac{1}{q_j}-\frac{\alpha}{m}, \frac{1}{r}=\sum_{j=1}^m\frac{1}{r_j}\in(\alpha,\alpha+1), \frac{1}{\tilde r_j}=\frac{1}{r_j}-\frac{\alpha}{m}. 
\end{equation*}
Let $\vec{w}=(w_1,\dots,w_m)\in\prod_{j=1}^m A_{r_j,\tilde r_j}(\R^{d})$, $\vec{v}=(v_1,\dots,v_m)\in\prod_{j=1}^m A_{q_j,\tilde q_j}(\\\R^{d})$. Then there exists $\vec{p}=(p_1,\dots,p_m)$ with $1<p_1,\dots,p_m<\infty$ satisfying $\frac{1}{p}=\sum_{j=1}^m\frac{1}{p_j}\in(\alpha,\alpha+1)$, $\frac{1}{\tilde p_j}=\frac{1}{p_j}-\frac{\alpha}{m}$, $\frac{1}{q_j}-\frac{1}{\tilde q_j}=\frac{1}{p_j}-\frac{1}{\tilde p_j}=\frac{1}{r_j}-\frac{1}{\tilde r_j}$, $\vec{u}=(u_1,\dots,u_m)\in\prod_{j=1}^m A_{p_j,\tilde p_j}(\R^{d})$, $\theta\in(0,1)$ such that
\begin{equation*}
  \frac{1}{r_j}=\frac{1-\theta}{p_j}+\frac{\theta}{q_j},\qquad w_j=u_j^{1-\theta}v_j^{\theta},\qquad j=1,\dots,m,
\end{equation*}
and 
\begin{equation*}
  \frac{1}{r}=\frac{1-\theta}{p}+\frac{\theta}{q},\qquad \nu_{\vec{w}}=\nu_{\vec{u}}^{1-\theta}\nu_{\vec{v}}^{\theta}.
\end{equation*}
\end{lemma}

\begin{proof}
This follows by applying \cite[Lemma 4.7]{HL} to each component separately. 
\end{proof}

We now have the last missing ingredient of the proof of Theorem \ref{thm:main}:

\begin{proof}[Proof of Proposition \ref{prop:main}]
We prove the proposition in the case that the assumptions \ref{(1a)} are in force. The other cases are proved in a similar way.
We are given $\vec{q}=(q_1,q_2)$, $\vec{r}=(r_1,r_2)$, $\vec{s}=(s_1,s_2)$ with $s_j\in[1,\infty)$, $q_j,r_j\in(s_j,\infty)$ $(j=1,2)$ satisfying $\frac{1}{q}=\sum_{j=1}^2\frac{1}{q_j}<1$, $\frac{1}{r}=\sum_{j=1}^2\frac{1}{r_j}<1$, $\frac{1}{s}=\sum_{j=1}^2\frac{1}{s_j}$ and weights $\vec{w}=(w_1,w_2)\in A_{\vec{r}/\vec{s}}(\R^{2d})$, $\vec{v}=(v_1,v_2)\in A_{\vec{q}/\vec{s}}(\R^{2d})$. By Lemma \ref{lem:main1}, there are some $\vec{p}=(p_1,p_2)$ with $p_j\in(s_j,\infty)$ $(j=1,2)$ satisfying $\frac{1}{p}=\sum_{j=1}^2\frac{1}{p_j}<1$, weights $\vec{u}=(u_1,u_2)\in A_{\vec{p}/\vec{s}}(\R^{2d})$, and $\theta\in(0,1)$ such that
\begin{equation*}
  \frac{1}{r_j}=\frac{1-\theta}{p_j}+\frac{\theta}{q_j},\qquad w_j^{\frac{1}{r_j}}=u_j^{\frac{1-\theta}{p_j}}v_j^{\frac{\theta}{q_j}},\qquad j=1,2,
\end{equation*}
and 
\begin{equation*}
  \frac{1}{r}=\frac{1-\theta}{p}+\frac{\theta}{q},\qquad \nu_{\vec{w},\vec{r}}^\frac{1}{r}=\nu_{\vec{u},\vec{p}}^{\frac{1-\theta}{p}}\nu_{\vec{v},\vec{q}}^\frac{\theta}{q}.
\end{equation*}
By Theorem \ref{thm:SW}, we then have
\begin{equation*}
  [L^{p_j}(u_j),L^{q_j}(v_j)]_\theta=L^{r_j}(w_j),\qquad [L^{p}(\nu_{\vec{u},\vec{p}}),L^{q}(\nu_{\vec{v},\vec{q}})]_\theta=L^{r}(\nu_{\vec{w},\vec{r}}),
\end{equation*}
as we claimed.
\end{proof}

\section{Commutators of bilinear Calder\'on--Zygmund operators}\label{section 6}

In the remaining sections of this paper, we consider a number of applications of our abstract results to specific classes of operators. In our first application below, we consider bilinear Calder\'on--Zygmund operators which are defined as follows:

Let $T$ be a multilinear operator initially defined on the $m$-fold product of Schwartz spaces and taking values in the space of tempered distributions,
\begin{equation*}
  T:\testi(\R^d)\times\cdots\times\testi(\R^d)\to\testi'(\R^d).
\end{equation*}
Following \cite{GT}, we say that $T$ is an $m$-linear Calder\'on--Zygmund operator if, for some $1\leq q_j<\infty$, it extends to a bounded multilinear operator from $L^{q_1}(\R^d)\times\cdots\times L^{q_m}(\R^d)$ to $L^{q}(\R^d)$, where $\frac{1}{q}=\sum_{j=1}^m\frac{1}{q_j}$, and it has the representation
\begin{equation}\label{CZO}
  T(f_1,\dots,f_m)(x)=\int_{(\R^d)^m}K(x,y_1,\dots,y_m)f_1(y_1)\dots f_m(y_m)dy_1\dots dy_m,
\end{equation}
for all $x\notin\cap_{j=1}^m\supp f_j$, where the kernel $K$ satisfies the {\em size condition}
\begin{equation}\label{Kernel1}
  |K(x,y_1,\dots,y_m)|\lesssim\frac{1}{(\sum_{j=1}^m|x-y_j|)^{md}}
\end{equation}
for all $(x,y_1,\dots,y_m)\in(\R^d)^{m+1}$ with $x\neq y_j$ for some $j\in\{1,2,\dots,m\}$ and the {\em smoothness condition}
\begin{equation}\label{Kernel2}
  |K(x,\dots,y_j,\dots,y_m)-K(x,y_1,\dots,z,\dots,y_m)|\lesssim\frac{|y_j-z|^{\eps}}{(\sum_{j=1}^m|x-y_j|)^{md+\eps}},
\end{equation}
for some $\eps>0$ and all $1\leq j\leq m$, whenever $|y_j-z|\leq\frac{1}{2}\max_{1\leq j\leq m}|x-y_j|$. Also, $T$ is called the multilinear Calder\'on--Zygmund operator associated with the kernel $K$.

In \cite{BC, KLOPTG, PPTG}, the following weighted boundedness results about the bilinear Calder\'on--Zygmund operator associated with kernel $K$ and its commutators were obtained:

\begin{theorem}[\cite{BC}, Theorem 1.2, \cite{KLOPTG}, Corollary 3.9,
Theorem 3.18 and \cite{PPTG}, Theorem 1.1]\label{thm:bdd. of CZO}
Suppose that $\vec{w}\in A_{\vec{p}}(\R^{2d})$ and $\vec{b}\in\BMO(\R^d)^2$ with $\frac{1}{p}=\sum_{j=1}^2\frac{1}{p_j}$, $1<p_j<\infty$, $j=1,2$, $p\in(1,\infty)$. Then $T$ and $[T,\vec{b}]_{\alpha}$ for each $\alpha\in\{(0,1),(1,0),(1,1)\}$ are bounded bilinear operators from $L^{p_1}(w_1)\times L^{p_2}(w_2)$ to $L^p(\nu_{\vec w,\vec{p}})$ under either of the following cases:
\begin{enumerate}[(1)]
  \item\namedlabel{thm:KLOPTG-PPTG bdd.}{(1)} $T$ is a bilinear Calder\'on--Zygmund operator associated with kernel $K$ satisfying \eqref{CZO}, \eqref{Kernel1}, \eqref{Kernel2}.
  \item\namedlabel{thm:BC bdd.}{(2)} $T$ is a bilinear singular integral operator associated with a kernel $K$ in the sense of \eqref{CZO} and satisfying \eqref{Kernel1}, and 
\begin{enumerate} 
  \item[(a)] $T$ is bounded from
  \begin{equation}\label{CZO1}
  L^1(\R^d)\times L^1(\R^d)\to L^{1/2,\infty}(\R^d),
  \end{equation}
  where $L^{1/2,\infty}(\R^d)$ is the weak $L^{1/2}$ space.
  \item[(b)] for $x,z,y_1,y_2\in\R^d$ with $8|x-z|<\min_{1\leq j\leq 2}|x-y_j|$,
  \begin{equation}\label{Kernel4}
  |K(x,y_1,y_2)-K(z,y_1,y_2)|\lesssim\frac{\tau^{\eps}}{(\sum_{j=1}^2|x-y_j|)^{2d+\eps}},
  \end{equation}
\end{enumerate}
where $\tau$ is a number such that $2|x-z|<\tau$ and $4\tau<\min_{1\leq j\leq 2}|x-y_j|$.
\end{enumerate}
\end{theorem}

\begin{remark}
The boundedness of $T$ in \cite{BC} is not explicitly stated as a theorem but it is implicitly contained in the proof of the boundedness of the commutator. As explained in \cite[Theorem 1.2]{BC}, \cite[Corollary 3.9 and Theorem 3.18]{KLOPTG} and \cite[Theorem 1.1]{PPTG} all the previous bounds for iterated commutators hold for $0<p<\infty$.
Also, it was pointed out in \cite[Proof of Theorem 1]{HZ} (see also \cite[Propositions 2.3, 4.1 and Remark 4.2]{DGGLY}) that the condition \eqref{Kernel4} is weaker than, and indeed a consequence of \eqref{Kernel2}.
\end{remark}

The compactness of the commutator $[T,\vec{b}]_{\alpha}$ was considered by B\'enyi--Torres \cite{BT2013} and Bu--Chen \cite{BC} in the unweighted case:

\begin{theorem}[\cite{BT2013}, Theorem 1 and \cite{BC}, Theorem 1.1]\label{thm:comp. of CZO}
Suppose that $\vec{b}\in\CMO(\R^d)^2$, $\frac{1}{p_1}+\frac{1}{p_2}=\frac{1}{p}$, $1<p_1,p_2<\infty$ and $1< p<\infty$. Then $[T,\vec{b}]_{\alpha}$ is compact from $L^{p_1}(\R^d)\times L^{p_2}(\R^d)$ to $L^{p}(\R^d)$ in each of the cases \ref{thm:KLOPTG-PPTG bdd.} and \ref{thm:BC bdd.} of Theorem \ref{thm:bdd. of CZO}.
\end{theorem}

A combination of the above Theorems \ref{thm:bdd. of CZO} and \ref{thm:comp. of CZO} with our main Theorem \ref{thm:main} recovers and improves the following results of B\'enyi et al. \cite[Theorems 3.1 and 3.2]{BDMT2015} and Bu--Chen \cite[Theorem 1.1]{BC}, lifting their additional assumption that $\nu_{\vec w,\vec{p}}\in A_{p}(\R^d)$:

\begin{theorem}\label{thm:weighted comp. of CZO}
Assume $\vec{b}\in\CMO(\R^d)^2$, $p_1,p_2\in(1,\infty)$, $p\in(1,\infty)$ such that $1/p=1/p_1+1/p_2$ and  $\vec{w}=(w_1,w_2)\in A_{\vec{p}}(\R^{2d})$. Then $[T,\vec{b}]_{\alpha}$ is compact from  $L^{p_1}(w_1)\times L^{p_2}(w_2)$ to $L^p(\nu_{\vec w,\vec{p}})$ in each of the cases \ref{thm:KLOPTG-PPTG bdd.} and \ref{thm:BC bdd.} of Theorem \ref{thm:bdd. of CZO}.
\end{theorem}

\begin{proof}
We prove the theorem  in the case that the assumptions \ref{thm:KLOPTG-PPTG bdd.} of Theorem \ref{thm:bdd. of CZO} are in force. The other case is proved in a similar way.
We verify the assumptions \ref{(1a)} of Theorem \ref{thm:main} for $[T,\vec{b}]_{\alpha}$ for each $\alpha\in\{(0,1),(1,0),(1,1)\}$ in place of $T$: By Theorem \ref{thm:bdd. of CZO}, $[T,\vec{b}]_{\alpha}$ is a bounded operator from $L^{q_1}(u_1)\times L^{q_2}(u_2)$ to $L^q(\nu_{\vec u,\vec{q}})$ for all $\vec{q}=(q_1,q_2)\in(1,\infty)^2$, $q=q_1 q_2/(q_1+q_2)>1$ and all $\vec{u}\in A_{\vec{q}}(\R^{2d})$. By Theorem \ref{thm:comp. of CZO}, $[T,\vec{b}]_{\alpha}$ is compact from $L^{r_1}(\R^d)=L^{r_1}(v_1)\times L^{r_2}(\R^d)=L^{r_2}(v_2)$ to $L^{r}(\R^d)=L^r(\nu_{\vec v,\vec{r}})$ with $\vec{v}=(v_1,v_2)\equiv(1,1)\in A_{\vec{r}}(\R^{2d})$ and $\nu_{\vec v,\vec{r}}\equiv 1$. Thus Theorem \ref{thm:main} applies to give the compactness of $[T,\vec{b}]_{\alpha}$ from $L^{p_1}(w_1)\times L^{p_2}(w_2)$ to $L^p(\nu_{\vec w,\vec{p}})$ for all $\vec{p}=(p_1,p_2)\in(1,\infty)^2$, $p=p_1 p_2/(p_1+p_2)>1$ and all $\vec{w}\in A_{\vec{p}}(\R^{2d})$.
\end{proof}

The proofs in \cite{BDMT2015,BC} were based on considering smooth truncation operators (\cite{ClopCruz, KL}) and verifying a weighted Fr\'echet--Kolmogorov criterion \cite{ClopCruz}. We avoid these considerations and obtain a more general theorem than these earlier approaches. However, the very recent work of Cao--Olivo--Yabuta \cite{COY} achieves a further generalisation (lifting also the assumption that $p>1$) by further developing the approach based on a weighted Fr\'echet--Kolmogorov criterion. It might be interesting to investigate whether the full scope of the results of \cite{COY} could be recovered avoiding this criterion.

\section{Commutators of bilinear fractional integral operators}\label{section 7}

In this section we apply Theorem \ref{thm:main} to the commutator $[I_{\beta},\vec{b}]_{\alpha}$, where $\alpha\in\{(0,1),(1,0),(1,1)\}$ and, given $0<\beta<2d$, the bilinear fractional integral operator $I_{\beta}$ is defined by
\begin{equation*}
  I_{\beta}(f_1,f_2)(x)=\int_{\R^{2d}} 
\frac{1}{(|x-y_1|^2+|x-y_2|^2)^{2d-\beta}} f_1(y_1)f_2(y_2)dy_1 dy_2. 
\end{equation*}

Chen--Wu \cite{CW} and Moen \cite{Moen} obtained the following weighted boundedness results of the bilinear fractional integral operator and its commutators:

\begin{theorem}[\cite{CW}, Theorems 1.4, 1.7 and \cite{Moen}, Theorem 3.5]\label{thm:weigh. bdd. Ib}
Suppose that $0<\beta<2d$, $\vec{b}\in\BMO(\R^d)^2$ and $1<p_1,p_2<\infty$ are exponents with $1/p=1/p_1+1/p_2, 1/2<p<d/\beta$ and $q$ is the exponent defined by $1/q=1/p-\beta/d$. Then $I_{\beta}$ and $[I_{\beta},\vec{b}]_{\alpha}$ for each  $\alpha\in\{(0,1),(1,0),(1,1)\}$ are bounded bilinear operators from $L^{p_1}(w_1^{p_1})\times L^{p_2}(w_2^{p_2})$ to $L^q(\nu_{\vec{w}}^q) $ for all $\vec{w}=(w_1,w_2)\in A_{\vec{p}, q}(\R^{2d})$.
\end{theorem}

The compactness of the commutator $[I_{\beta},\vec{b}]_{\alpha}$ was considered by Chaffee--Torres \cite{CT} and Wang--Zhou--Teng \cite{WZT} in the unweighted case:

\begin{theorem}[\cite{CT}, Theorem 3.1 (i) and (iii) and \cite{WZT}, Theorem 1.2 (A1) and (A3)]\label{thm:unw. comp. Ib}
Let $1<p_1,p_2<\infty$, $\vec{p}=(p_1,p_2)$, $\frac{1}{p}=\frac{1}{p_1}+\frac{1}{p_2}$, $0<\beta<2d$, $\frac{\beta}{d}<\frac{1}{p_1}+\frac{1}{p_2}$, and $q$ such that $\frac{1}{q}=\frac{1}{p_1}+\frac{1}{p_2}-\frac{\beta}{d}$ and $1<p,q<\infty$. If $\vec{b}=(b,b)\in\CMO(\R^d)^2$, then $[I_{\beta},\vec{b}]_{\alpha}$ for each $\alpha\in\{(0,1),(1,0),(1,1)\}$ is compact from $L^{p_1}(\R^d) \times L^{p_2}(\R^d)$ to $L^q(\R^d)$.
\end{theorem}

Thus, by combining the verification of the assumptions \ref{(2c)} of Theorem \ref{thm:main}, Theorems \ref{thm:weigh. bdd. Ib} and \ref{thm:unw. comp. Ib} we can now recover and improve the following results of Chaffee--Torres \cite[Theorem 3.1 (ii)]{CT} and Wang--Zhou--Teng \cite[Theorem 1.2 (A2)]{WZT}, lifting their additional assumption that $w_1^{\frac{p_1 q}{p}},w_2^{\frac{p_2 q}{p}}\in A_{p}(\R^d)$:

\begin{theorem}\label{thm:Ib}
Let $1<p_1,p_2<\infty$, $\vec{p}=(p_1,p_2)$, $\frac{1}{p}=\frac{1}{p_1}+\frac{1}{p_2}$, $0<\beta<2d$, $\frac{\beta}{d}<\frac{1}{p_1}+\frac{1}{p_2}$, and $q$ such that $\frac{1}{q}=\frac{1}{p_1}+\frac{1}{p_2}-\frac{\beta}{d}$ and $1<p,q<\infty$. If $\vec{b}=(b,b)\in\CMO(\R^d)^2$, then $[I_{\beta},\vec{b}]_{\alpha}$ for each $\alpha\in\{(0,1),(1,0),(1,1)\}$ is compact from $L^{p_1}(w_1^{p_1})\times L^{p_2}(w_2^{p_2})$ to $L^q(\nu_{\vec{w}}^q)$ for all $\vec{w}=(w_1, w_2) \in A_{\vec{p}, q}(\R^{2d})$.  
\end{theorem}

\begin{proof}
We verify the assumptions \ref{(2c)} of Theorem \ref{thm:main} for $[I_{\beta},\vec{b}]_{\alpha}$ for each $\alpha\in\{(0,1),(1,0),(1,1)\}$ in place of $T$: By Theorem \ref{thm:weigh. bdd. Ib}, $[I_{\beta},\vec{b}]_{\alpha}$ is a bounded operator from $L^{s_1}(u_1)\times L^{s_2}(u_2)$ to $L^{s^*}(\nu_{\vec u}^{s^*})$ for all $\vec{s}=(s_1,s_2)\in(1,\infty)^2$, $\frac{1}{s}=\frac{1}{s_1}+\frac{1}{s_2}<1$, $s^*>1$ such that $\frac{1}{s^*}=\frac{1}{s_1}+\frac{1}{s_2}-\frac{\beta}{d}$ and all $\vec{u}\in A_{\vec{s}, s^*}(\R^{2d})$. By Theorem \ref{thm:unw. comp. Ib}, $[I_{\beta},\vec{b}]_{\alpha}$ is compact from $L^{r_1}(\R^d)=L^{r_1}(v_1^{r_1})\times L^{r_2}(\R^d)=L^{r_2}(v_2^{r_2})$ to $L^{r^*}(\R^d)=L^{r^*}(\nu_{\vec{r}}^{r^*})$ with $\vec{v}=(v_1,v_2)\equiv(1,1)\in  A_{\vec{r}, r^*}(\R^{2d})$ and $\nu_{\vec{r}}\equiv 1$. Thus Theorem \ref{thm:main} applies to give the compactness of $[I_{\beta},\vec{b}]_{\alpha}$ from $L^{p_1}(w_1^{p_1})\times L^{p_2}(w_2^{p_2})$ to $L^q(\nu_{\vec{w}}^q)$ for all $\vec{p}=(p_1,p_2)\in (1,\infty)^2$, $\frac{1}{p}=\frac{1}{p_1}+\frac{1}{p_2}<1$, $q>1$ such that $\frac{1}{p}-\frac{1}{q}=\frac{1}{s}-\frac{1}{s^*}=\frac{1}{r}-\frac{1}{r^*}=\frac{\beta}{d}$ and all $\vec{w}=(w_1, w_2)\in A_{\vec{p}, q}(\R^{2d})$. 
\end{proof}

The proofs in \cite{CT,WZT} were based on considering smooth truncations of $I_\beta$ (see \cite{BDMT} and the references therein) and verifying the weighted Fr\'echet--Kolmogorov criterion \cite{ClopCruz}. We obtain a more general result by avoiding this criterion, but Cao--Olivo--Yabuta \cite{COY} achieve a further generalization by an approach again based on the weighted  Fr\'echet--Kolmogorov criterion.

\section{Commutators of bilinear Fourier multipliers}\label{section 8}

In this section, we apply Theorem \ref{thm:main} to the commutators of bilinear Fourier multipliers (first studied by Coifman and Meyer \cite{CM:audela}) which satisfy certain Sobolev regularity conditions. 

Given $s \in \R$ and $\vec s=(s_1,s_2)\in\R^2$, the Sobolev spaces $H^s(\R^{2d})$ and $H^{\vec s}(\R^{2d})$ are defined by the norms
\begin{equation*} 
\begin{split}
  \|f\|_{H^{s}(\R^{2d})} &=\bigg(\int_{\R^{2d}} (1+|\xi_1|^2+|\xi_2|^2)^{s}|\widehat{f}(\xi_1,\xi_2)|^2 d\xi_1 d\xi_2\bigg)^{\frac{1}{2}},  \\
  \|f\|_{H^{\vec s}(\R^{2d})} &=\bigg(\int_{\R^{2d}} (1+|\xi_1|^2)^{s_1}
  (1+|\xi_2|^2)^{s_2} |\widehat{f}(\xi_1,\xi_2)|^2 d\xi_1 d\xi_2\bigg)^{\frac{1}{2}},
\end{split}
\end{equation*}
where $\widehat{f}$ denotes the Fourier transform of $f$. 
Let $\Phi \in \testi(\R^{2d})$ satisfy 
\begin{equation*}
\begin{cases}
  \supp(\Phi) \subset \big\{(\xi_1,\xi_2):\frac{1}{2} \leq |\xi_1|+|\xi_2|\leq 2\big\};\\
  \sum_{j \in \Z} \Phi(2^{-j}\xi_1,2^{-j}\xi_2)=1 \quad \text{for all}\quad (\xi_1,\xi_2) \in \R^{2d} \setminus\{0\}.
\end{cases}
\end{equation*}
For $\sigma\in L^{\infty}(\R^{2d})$, we denote $\sigma_j(\xi_1,\xi_2)=\Phi(\xi_1,\xi_2)\sigma(2^j\xi_1,2^j\xi_2)$ for $j\in\Z$.
The bilinear Fourier multiplier $T_{\sigma}$ with symbol $\sigma$ is defined by
\begin{equation*}
  T_{\sigma}(f_1,f_2)(x)=\int_{\R^{2d}}\sigma(\xi_1,\xi_2)
  \widehat{f}_1(\xi_1)\widehat{f}_2(\xi_2)e^{2\pi i x (\xi_1+\xi_2)} d\xi_1 d\xi_2,
\end{equation*}
for $f_1,f_2\in\testi(\R^d)$.

Fujita--Tomita \cite{FT12}, Jiao \cite{Jiao} and Zhou--Li \cite{ZL} obtained the following weighted boundedness results for $T_{\sigma}$ and its commutators:

\begin{theorem}[\cite{FT12}, Theorem 6.2, \cite{Jiao} and \cite{ZL}, Theorem 1]\label{thm:Fourier multi. bdd.}
Let $\vec{b}\in\BMO(\R^d)^2$. The operators $T_{\sigma}$ and $[T_{\sigma},\vec{b}]_{\alpha}$ for each $\alpha\in\{(0,1),(1,0)\}$ are bounded bilinear operators from $L^{p_1}(w_1) \times L^{p_2}(w_2)$ to $L^p(\nu_{\vec{w},\vec{p}})$, where  $\frac{1}{p}=\frac{1}{p_1}+\frac{1}{p_2}<1$, under either of the following cases:
\begin{enumerate}[(1)]
  \item\namedlabel{thm:Fourier multi. bdd. 1}{(1)} 
  \begin{enumerate}
  \item $\sigma$ satisfies $\sup_{j \in \Z} \|\sigma_j\|_{H^{s}(\R^{2d})}<\infty$ with $s\in (d,2d]$,
  \item $p_j\in(t_j,\infty)$ for some $t_j\in[1,2)$ such that $\frac{1}{t_1}+\frac{1}{t_2}=\frac{s}{d}$,  and
  \item $\vec{w}=(w_1,w_2)\in A_{\vec{p}/ \vec{t}}(\R^{2d})$.
\end{enumerate}
  \item\namedlabel{thm:Fourier multi. bdd. 2}{(2)} 
\begin{enumerate}
  \item $\sigma$ satisfies $\sup_{j \in \Z}  \|\sigma_j\|_{H^{\vec s}(\R^{2d})}<\infty$ with $\vec s=(s_1,s_2) \in (d/2,d]^2$,
  \item $p_j>d/s_j$ and 
  \item $\vec{w}=(w_1,w_2) \in A_{p_1 s_1/d}(\R^d) \times A_{p_2 s_2/d}(\R^d)$.
\end{enumerate}
\end{enumerate}
\end{theorem}

\begin{proof}
In the cases \ref{thm:Fourier multi. bdd. 1} and \ref{thm:Fourier multi. bdd. 2} the boundedness of the operator $T_{\sigma}$ is contained in \cite{Jiao} and \cite[Theorem 6.2]{FT12} respectively. The boundedness of the commutators $[T_{\sigma},\vec{b}]_{\alpha}$ in the case \ref{thm:Fourier multi. bdd. 1} follows by combining the boundedness of the operator $T_{\sigma}$ in \cite{Jiao} with Theorem \ref{thm: Ap/s bdd.}. In the case \ref{thm:Fourier multi. bdd. 2} the boundedness of the commutators $[T_{\sigma},\vec{b}]_{\alpha}$ is contained in \cite[Theorem 1]{ZL}.
\end{proof}

\begin{remark}
As mentioned in \cite[Theorem 6.2]{FT12} and \cite{Jiao}, in both of the cases \ref{thm:Fourier multi. bdd. 1} and \ref{thm:Fourier multi. bdd. 2} of Theorem \ref{thm:Fourier multi. bdd.} the boundedness of the bilinear operator $T_{\sigma}$ holds for $0<p<\infty$.  
\end{remark}

Compactness of the commutator $[T_{\sigma},\vec{b}]_{\alpha}$ in the unweighted case was considered by Hu \cite{Hu14, Hu17}:

\begin{theorem}[\cite{Hu14}, Theorem 1.1 and \cite{Hu17}, Theorem 1.1]\label{thm:Fourier multi. comp.}
Suppose that $\vec{b}\in\CMO(\R^d)^2$. Then $[T_{\sigma},\vec{b}]_{\alpha}$ for each $\alpha\in\{(0,1),(1,0)\}$ is compact from $L^{p_1}(\R^d) \times L^{p_2}(\R^d)$ to $L^p(\R^d)$ in each of the cases \ref{thm:Fourier multi. bdd. 1} and \ref{thm:Fourier multi. bdd. 2} of Theorem \ref{thm:Fourier multi. bdd.}.
\end{theorem}

By combining the above Theorems \ref{thm:Fourier multi. bdd.} and \ref{thm:Fourier multi. comp.} with our main Theorem \ref{thm:main}, we can now recover and improve the result of Hu \cite[Theorem 1.1]{Hu17}, lifting their assumption that $\nu_{\vec w,\vec{p}}\in A_{p}(\R^d)$, and the result of Zhou--Li \cite[Theorem 2]{ZL}:

\begin{theorem}\label{thm:Hu,ZL improved comp. result}
Suppose that $\vec{b}\in\CMO(\R^d)^2$. Then $[T_{\sigma},\vec{b}]_{\alpha}$ for each $\alpha\in\{(0,1),(1,0)\}$ is compact from $L^{p_1}(w_1) \times L^{p_2}(w_2)$ to $L^p(\nu_{\vec{w},\vec{p}})$ in each of the cases \ref{thm:Fourier multi. bdd. 1} and \ref{thm:Fourier multi. bdd. 2} of Theorem \ref{thm:Fourier multi. bdd.}.
\end{theorem}

\begin{proof}
We prove the theorem  in the case that the assumptions \ref{thm:Fourier multi. bdd. 1} of Theorem \ref{thm:Fourier multi. bdd.} are in force. The other case is proved in a similar way.
We verify the assumptions \ref{(1a)} of Theorem \ref{thm:main} for $[T_{\sigma},\vec{b}]_{\alpha}$ for each $\alpha\in\{(0,1),(1,0)\}$ in place of $T$: By Theorem \ref{thm:Fourier multi. bdd.},
$[T_{\sigma},\vec{b}]_{\alpha}$ is a bounded operator from $L^{q_1}(u_1)\times L^{q_2}(u_2)$ to $L^q(\nu_{\vec u,\vec{q}})$ for all $\vec{q}=(q_1,q_2)\in(t_j,\infty)^2$, $q>1$ and all $\vec{u}\in A_{\vec{q}/\vec{t}}(\R^{2d})$. By Theorem \ref{thm:Fourier multi. comp.}, $[T_{\sigma},\vec{b}]_{\alpha}$ is compact from $L^{r_1}(\R^d)=L^{r_1}(v_1)\times L^{r_2}(\R^d)=L^{r_2}(v_2)$ to $L^{r}(\R^d)=L^r(\nu_{\vec v,\vec{r}})$ with $\vec{v}=(v_1,v_2)\equiv(1,1)\in A_{\vec{r}}(\R^{2d})$ and $\nu_{\vec v,\vec{r}}\equiv1$. Thus Theorem \ref{thm:main} applies to give the compactness of $[T_{\sigma},\vec{b}]_{\alpha}$ from $L^{p_1}(w_1)\times L^{p_2}(w_2)$ to $L^p(\nu_{\vec w,\vec{p}})$ for all $\vec{p}=(p_1,p_2)\in(t_j,\infty)^2$, $p>1$ and all $\vec{w}\in A_{\vec{p}/\vec{t}}(\R^{2d})$. If we work under the assumptions \ref{thm:Fourier multi. bdd. 2} of Theorem \ref{thm:Fourier multi. bdd.} then we verify the assumptions \ref{(1b)} of Theorem \ref{thm:main}.
\end{proof}

The proof in \cite{Hu17} was based on the idea of introducing a new subtle bi(sub)linear maximal operator to control the commutators $[T_{\sigma},\vec{b}]_{\alpha}$. As in the cases of the commutators of bilinear Calder\'on--Zygmund and fractional integral operators both of the original proofs in \cite{Hu17,ZL} relied on verifying the weighted Fr\'echet--Kolmogorov criterion \cite{ClopCruz}, which is avoided by the argument above. Again, Cao--Olivo--Yabuta \cite{COY} obtain a further generalisation by developing the approach based on the weighted Fr\'echet--Kolmogorov criterion.


\end{document}